\documentclass{amsart}
\usepackage{amsmath,amscd}
\usepackage{amssymb}
\usepackage{bbm}
\usepackage{enumerate}

\newtheorem{theorem}{Theorem}[section] 
\newtheorem{lemma}[theorem]{Lemma}
\newtheorem{corollary}[theorem]{Corollary}
\newtheorem{proposition}[theorem]{Proposition}

\theoremstyle{definition}
\newtheorem{definition}[theorem]{Definition}

\theoremstyle{remark}

\numberwithin{equation}{section}

\begin{document}

\title{Wave Front Sets of Reductive Lie Group Representations}

\author{Benjamin Harris}
\address{Department of Mathematics, Oklahoma State University, Stillwater, Oklahoma 74078}
\email{Benjamin.Harris@math.okstate.edu}
\thanks{The first author was an NSF VIGRE postdoc at LSU while this research was conducted.}

\author{Hongyu He}
\address{Department of Mathematics, Louisiana State University, Baton Rouge, Louisiana 70803}
\email{hongyu@math.lsu.edu}

\author{Gestur \'{O}lafsson}
\address{Department of Mathematics, Louisiana State University, Baton Rouge, Louisiana 70803}
\email{olafsson@math.lsu.edu}
\thanks{The third author was supported by NSF grant 1101337 while this research was conducted.}

\subjclass[2010]{22E46, 22E45, 43A85}

\date{March 19, 2015}


\keywords{Wave Front Set, Singular Spectrum, Analytic Wave Front Set, Reductive Lie Group, Induced Representation, Tempered Representation, Branching Problem, Discrete Series, Reductive Homogeneous Space}

\begin{abstract}
If $G$ is a Lie group, $H\subset G$ is a closed subgroup, and $\tau$ is a unitary representation of $H$, then the authors give a sufficient condition on $\xi\in i\mathfrak{g}^*$ to be in the wave front set of $\operatorname{Ind}_H^G\tau$. In the special case where $\tau$ is the trivial representation, this result was conjectured by Howe. If $G$ is a real, reductive algebraic group and $\pi$ is a unitary representation of $G$ that is weakly contained in the regular representation, then the authors give a geometric description of $\operatorname{WF}(\pi)$ in terms of the direct integral decomposition of $\pi$ into irreducibles. Special cases of this result were previously obtained by Kashiwara-Vergne, Howe, and Rossmann. The authors give applications to harmonic analysis problems and branching problems.
\end{abstract}

\maketitle

\section{Introduction}
If $u$ is a distribution on a smooth manifold $X$, then the \emph{wave front set} of $u$, denoted $\operatorname{WF}(u)$, is a closed subset of $iT^*X$ that microlocally measures the smoothness of the distribution $u$ (see Section \ref{sec:def_wf} for a definition). Similarly, if $\zeta$ is a hyperfunction on an analytic manifold $Y$, then the \emph{singular spectrum} of $\zeta$, denoted $\operatorname{SS}(\zeta)$, is a closed subset of $iT^*Y$ that microlocally measures the analyticity of the hyperfunction $\zeta$ (see Section \ref{sec:def_wf} for a definition). The singular spectrum is also called the \emph{analytic wave front set}.

Suppose $G$ is a Lie group, $(\pi,V)$ is a unitary representation of $G$, and $(\cdot,\cdot)$ is the inner product on the Hilbert space $V$. Then the \emph{wave front set of} $\pi$ and the \emph{singular spectrum of} $\pi$ are defined by
$$\operatorname{WF}(\pi)=\overline{\bigcup_{u,v\in V}\operatorname{WF}_e(\pi(g)u,v)},\ \  \operatorname{SS}(\pi)=\overline{\bigcup_{u,v\in V}\operatorname{SS}_e(\pi(g)u,v)}.$$
Here the subscript $e$ means we are only considering the piece of the wave front set (or the singular spectrum) of the matrix coefficient $(\pi(g)u,v)$ in the fiber over the identity in $iT^*G$. 

In the case where $G$ is compact, a notion equivalent to the singular spectrum of a unitary representation was introduced by Kashiwara and Vergne on the top of page 192 of \cite{KV}. This notion was later used by Kobayashi in \cite{Ko2} to prove a powerful sufficient condition for discrete decomposability. Our definition of the wave front set of a representation is equivalent to $i$ times the definition of $\operatorname{WF}^0(\pi)$ first introduced by Howe in \cite{How} (see Proposition \ref{howeusequiv} for the equivalence of the two definitions). The wave front set and singular spectrum of a representation are always closed, invariant cones in $i\mathfrak{g}^*$, the dual of the Lie algebra of $G$. 

Suppose $G$ is a Lie group, $H\subset G$ is a closed subgroup, and $\tau$ is a unitary representation of $H$. Then we may form the unitarily induced representation $\operatorname{Ind}_H^G\tau$, which is a unitary representation of $G$ (See Section \ref{sec:wf_induced} for the definition). Let $\mathfrak{g}$ (resp. $\mathfrak{h}$) denote the Lie algebra of $G$ (resp. $H$), and let $q\colon i\mathfrak{g}^*\rightarrow i\mathfrak{h}^*$ be the pullback of the inclusion. If $S\subset i\mathfrak{h}^*$ is a subset, we will denote 
$$\operatorname{Ind}_H^G S=\overline{\operatorname{Ad}^*(G)\cdot q^{-1}(S)}$$
and we will call this the set induced by $S$ from $i\mathfrak{h}^*$ to $i\mathfrak{g}^*$.

\begin{theorem} \label{inducedintro} Suppose $G$ is a Lie group, $H\subset G$ is a closed subgroup, and $\tau$ is a unitary representation of $H$. Then 
$$\operatorname{WF}(\operatorname{Ind}_H^G\tau)\supset \operatorname{Ind}_H^G\operatorname{WF}(\tau)$$
and 
$$\operatorname{SS}(\operatorname{Ind}_H^G\tau)\supset \operatorname{Ind}_H^G\operatorname{SS}(\tau).$$
\end{theorem}

When $\tau=\mathbbm{1}$ is the trivial representation, we have $\operatorname{WF}(\mathbbm{1})=\{0\}$ and we obtain
$$\operatorname{WF}(\operatorname{Ind}_H^G \mathbbm{1})\supset \overline{\operatorname{Ad}^*(G)\cdot i(\mathfrak{g}/\mathfrak{h})^*}\supset i(\mathfrak{g}/\mathfrak{h})^*.$$
This special case was conjectured by Howe on page 128 of \cite{How}. In the case where $G$ is compact, the equality $\operatorname{SS}(\operatorname{Ind}_H^G\tau)=\operatorname{Ind}_H^G\operatorname{SS}(\tau)$ was obtained by Kashiwara and Vergne in Proposition 5.4 of \cite{KV}. In the case where $G$ is a connected semisimple Lie group with finite center, $H=P=MAN\subset G$ is a parabolic subgroup, and $\tau$ is an irreducible, unitary representation of $MA$ extended trivially to $P$, the equality $\operatorname{WF}(\operatorname{Ind}_P^G\tau)=\operatorname{Ind}_P^G\operatorname{WF}(\tau)$ follows from work of Barbasch-Vogan (see page 39 of \cite{BV}) together with the principal results of \cite{R5}, \cite{SV}. Note that when $\Gamma\subset G$ is a discrete subgroup of a unimodular group $G$, we obtain $$\operatorname{WF}(L^2(G/\Gamma))=\operatorname{SS}(L^2(G/\Gamma))=i\mathfrak{g}^*.$$

Let $G$ be a real, reductive algebraic group. The irreducible representations occurring in the direct integral decomposition of $L^2(G)$ are called irreducible, tempered representations of $G$; we denote by $\widehat{G}_{\text{temp}}$ the subspace of the unitary dual consisting of these representations. This subspace is closed in the Fell topology on the unitary dual. To each irreducible tempered representation $\sigma$ of $G$, Duflo and Rossmann associated a finite union of coadjoint orbits $\mathcal{O}_{\sigma}\subset i\mathfrak{g}^*$ \cite{Du},\cite{R1},\cite{R2}. In the generic case, when $\sigma$ has regular infinitesimal character, $\mathcal{O}_{\sigma}$ is a single coadjoint orbit.

If $G$ is a real, reductive algebraic group and $(\pi,V)$ is a unitary representation of $G$, then we say $\pi$ is \emph{weakly contained in the regular representation} if $\operatorname{supp} \pi\subset \widehat{G}_{\text{temp}}$. 
For such a representation $\pi$, we define the \emph{orbital support} of $\pi$ to be
$$\mathcal{O}\operatorname{-}\operatorname{supp}\pi:=\bigcup_{\sigma\in \operatorname{supp} \pi}\mathcal{O}_{\sigma}.$$

If $W$ is a finite-dimensional vector space and $S\subset W$, then we define the \emph{asymptotic cone} of $S$ to be
$$\operatorname{AC}(S):=\{\xi\in W|\ \mathcal{C}\ \text{an\ open\ cone\ containing}\ \xi\Rightarrow \mathcal{C}\cap S\ \text{is\ unbounded}\}\cup \{0\}.$$
One notes that $\operatorname{AC}(S)$ is a closed cone. 

\begin{theorem} \label{regularintro} If $G$ is a real, reductive algebraic group and $\pi$ is weakly contained in the regular representation of $G$, then
$$\operatorname{SS}(\pi)=\operatorname{WF}(\pi)=\operatorname{AC}\left(\mathcal{O}\operatorname{-}\operatorname{supp}\pi\right).$$
\end{theorem}

When $G$ is compact and connected, an equivalent formula for $\operatorname{SS}(\pi)$ was obtained by Kashiwara and Vergne in Corollary 5.10 of \cite{KV}. Using similar ideas, Howe obtained the same formula for $\operatorname{WF}(\pi)$ when $G$ is compact in Proposition 2.3 of \cite{How}. Related results concerning wave front sets and compact groups $G$ appeared in \cite{HHK}. Finally, one can deduce the above formula for $\operatorname{WF}(\pi)$ when $\pi$ is irreducible from Theorems B and C of Rossmann's paper \cite{R5}. 

Note that when $K\subset G$ is a maximal compact subgroup of a semisimple Lie group, it is known that $L^2(G/K)$ is a direct integral of principal series representations (see \cite{HC58a}, \cite{HC58b}, \cite{HC65}, \cite{He70}, \cite{He73} for the original papers; see Section 1 of \cite{OS} for an expository introduction). Combining this knowledge with Theorem \ref{regularintro}, we obtain 
$$\operatorname{WF}(L^2(G/K))=\operatorname{SS}(L^2(G/K))=i\overline{\mathfrak{g}^*_{\text{hyp}}}=\overline{\operatorname{Ad}^*(G)\cdot i(\mathfrak{g}/\mathfrak{k})^*}.$$
Here $\mathfrak{g}^*_{\text{hyp}}$ denotes the set of hyperbolic elements in $\mathfrak{g}^*$.
\bigskip

Next, we consider two classes of applications of the above Theorems. First, suppose $G$ is a real, semisimple algebraic group and $H\subset G$ is a reductive subgroup. In Theorem 4.1 of \cite{BK}, Benoist and Kobayashi give a concrete and computable necessary and sufficient condition for $\operatorname{Ind}_H^G \mathbbm{1}=L^2(G/H)$ to be weakly contained in the regular representation. Putting together Theorems \ref{inducedintro} and \ref{regularintro}, we obtain the following Corollary.

\begin{corollary} \label{BKintro} If $G$ is a real, reductive algebraic group, $H\subset G$ is a closed subgroup, and $L^2(G/H)$ is weakly contained in the regular representation, then 
$$\operatorname{AC}\left(\mathcal{O}\operatorname{-}\operatorname{supp}L^2(G/H)\right)\supset \overline{\operatorname{Ad}^*(G)\cdot i(\mathfrak{g}/\mathfrak{h})^*}.$$
\end{corollary}

From Example 5.6 of \cite{BK}, we see that if $G=\operatorname{SO}(p,q)$ and $H=\prod_{i=1}^r \operatorname{SO}(p_i,q_i)$ with $p=\sum_{i=1}^r p_i$, $q=\sum_{i=1}^r q_i$, and $2(p_i+q_i)\leq p+q+2$ whenever $p_iq_i\neq 0$, then $L^2(G/H)$ is weakly contained in the regular representation. To the best of the authors' knowledge, Plancherel formulas are not known for the vast majority of these cases. An elementary computation shows that if in addition, $2p_i\leq p+1$ and $2q_i\leq q+1$ for every $i$ and $p+q>2$, then $$i\mathfrak{g}^*=\overline{\operatorname{Ad}^*(G)\cdot i(\mathfrak{g}/\mathfrak{h})^*}.$$
Corollary \ref{BKintro} now implies that $\operatorname{supp} L^2(G/H)$ is ``asymptotically equivalent to'' $\operatorname{supp}L^2(G)$ (we make this notion precise in Section \ref{sec:examples_apps}). In particular, suppose $p$ and $q$ are not both odd and $\mathcal{F}$ is one of the families of discrete series of $G=\operatorname{SO}(p,q)$ associated to a conjugacy class of Weyl chambers in the dual of a fundamental Cartan subalgebra of $\mathfrak{g}$. Then 
$$\operatorname{Hom}_G(\sigma,L^2(G/H))\neq \{0\}$$
for infinitely many different $\sigma\in \mathcal{F}$ (more details appear in Section \ref{sec:examples_apps}). 

In passing, we recall that Kobayashi previously obtained some partial results concerning the discrete spectrum of $L^2(G/H)$ for certain $G$ and $H$ when $G$ is reductive \cite{Ko5}. While there is some small amount of overlap between this paper and \cite{Ko5}, most of the results in each paper cannot be deduced from the results of the other paper.

Next, we utilize Theorem \ref{regularintro} together with an analogue of Theorem \ref{inducedintro} for restriction due to Howe in order to analyze branching problems for discrete series representations. First, we recall Howe's result (see page 124 of \cite{How}). If $\pi$ is a unitary representation of a Lie group $G$, $H\subset G$ is a closed subgroup, and $q:i\mathfrak{g}^*\rightarrow i\mathfrak{h}^*$ is the pullback of the inclusion, then $$\operatorname{WF}(\pi|_H)\supset q(\operatorname{WF}(\pi)).$$

\begin{corollary} \label{resintro1} Suppose $G$ is a real, reductive algebraic group, suppose $H\subset G$ is a closed reductive algebraic subgroup, and suppose $\pi$ is a discrete series representation of $G$. Let $\mathfrak{g}$ (resp. $\mathfrak{h}$) denote the Lie algebra of $G$ (resp. $H$), and let $q\colon i\mathfrak{g}^*\rightarrow i\mathfrak{h}^*$ be the pullback of the inclusion. Then 
$$\operatorname{AC}\left(\mathcal{O}\operatorname{-}\operatorname{supp}(\pi|_H)\right)\supset q(\operatorname{WF}(\pi))=q(\operatorname{AC}(\mathcal{O}_{\pi})).$$
\end{corollary}

Let $S$ be an exponential, solvable Lie group, let $T\subset S$ be a closed subgroup, and let $q\colon i\mathfrak{s}^*\rightarrow i\mathfrak{t}^*$ be the pullback of the inclusion of Lie algebras. Every irreducible, unitary representation $\pi\in \widehat{S}$ (resp. $\sigma\in \widehat{T}$) can be associated to a coadjoint orbit $\mathcal{O}_{\pi}$ (resp. $\mathcal{O}_{\sigma}$). Fujiwara proved that $\sigma$ occurs in the decomposition of $\pi|_H$ into irreducibles iff $\mathcal{O}_{\sigma}\subset q(\mathcal{O}_{\pi})$ \cite{Fu}. 
The above Corollary can be viewed as (half of) an asympototic version of Fujiwara's statement for reductive groups.

We take note of a special case of Corollary \ref{resintro1} that may be of particular interest.

\begin{corollary} \label{resintro2} Suppose $G$ is a real, reductive algebraic group, $H\subset G$ is a reductive algebraic subgroup, and $\pi$ is a discrete series representation of $G$. Let $\mathfrak{g}$ (resp. $\mathfrak{h}$) denote the Lie algebra of $G$ (resp. $H$), and let $q\colon i\mathfrak{g}^*\rightarrow i\mathfrak{h}^*$ be the pullback of the inclusion. If $\pi|_H$ is a Hilbert space direct sum of irreducible representations of $H$, then 
$$q(\operatorname{WF}(\pi))\subset i\overline{\mathfrak{h}^*_{\text{ell}}}.$$
Here $i\mathfrak{h}^*_{\text{ell}}\subset i\mathfrak{h}^*$ denotes the subset of elliptic elements. 
\end{corollary}

Let $G$ be a real, reductive algebraic group with Lie algebra $\mathfrak{g}$, let $K\subset G$ be a maximal compact subgroup with Lie algebra $\mathfrak{k}$ and complexification $K_{\mathbb{C}}$, and let $\mathcal{N}(\mathfrak{g}_{\mathbb{C}}/\mathfrak{k}_{\mathbb{C}})^*$ denote the set of nilpotent elements of $\mathfrak{g}_{\mathbb{C}}^*$ in $(\mathfrak{g}_{\mathbb{C}}/\mathfrak{k}_{\mathbb{C}})^*$. In \cite{V}, Vogan introduced the \emph{associated variety} of an irreducibe, unitary representation $\pi\in \widehat{G}$, denoted $\operatorname{AV}(\pi)$. It is a closed, $K$ invariant subset of $\mathcal{N}(\mathfrak{g}_{\mathbb{C}}/\mathfrak{k}_{\mathbb{C}})^*$. For an irreducible, unitary representation $\pi$ of $G$, there is a known procedure for producing $\operatorname{AV}(\pi)$ from $\operatorname{WF}(\pi)$ and vice versa \cite{SV}, \cite{R5}, \cite{BV}. In particular, these notions give equivalent information about $\pi$.

Now, suppose $H\subset G$ is a real, reductive algebraic subgroup such that $K\cap H\subset H$ is a maximal compact subgroup. Let $(\pi,V)$ be an irreducible, unitary representation of $G$, and let $V_K$ be the set of $K$ finite vectors of $V$. Note $V_K$ is a $\mathfrak{g}$ module. In Corollary 3.4 of \cite{Ko3} (see also Corollary 5.8 of \cite{Ko4}), Kobayashi showed that if $V_K|_{\mathfrak{h}}$ is discretely decomposable as an $\mathfrak{h}$ module, then $$q(\operatorname{AV}(\pi))\subset \mathcal{N}(\mathfrak{h}_{\mathbb{C}}/(\mathfrak{h}_{\mathbb{C}}\cap \mathfrak{k}_{\mathbb{C}}))^*.$$
Here $q\colon \mathfrak{g}_{\mathbb{C}}^*\rightarrow \mathfrak{h}_{\mathbb{C}}^*$ is the pullback of the inclusion. Corollary \ref{resintro2} can be viewed as an analogue of Kobayashi's statement with $\operatorname{AV}(\pi)$ replaced by $\operatorname{WF}(\pi)$ and in the special case where $\pi$ is a discrete series representation. 

We end this introduction by remarking that an earlier version of this paper proved all of the main results for reductive Lie groups of Harish-Chandra class rather than real, reductive algebraic groups. The authors made this switch in order to simplify the exposition.

\section{The Definition of the Wave Front Set}
\label{sec:def_wf}

In this section, we give definitions of the wave front set of a distribution, the singular spectrum of a hyperfunction, the wave front set of a unitary Lie group representation, and the singular spectrum of a unitary Lie group representation. In addition, we collect a few facts about these objects to be used later in the paper.

\bigskip

There are two types of distributions (resp. tempered distributions) on a smooth manifold $X$. First, there is the set of generalized measures, which is the set of continuous linear functionals on the space of smooth, compactly supported functions on $X$. Second, there is the set of generalized functions, which is the set of continuous linear functionals on the space of smooth, compactly supported densities on $X$. In the special case where $X$ is a real, finite dimensional vector space $W$, we can similarly talk about tempered generalized measures and tempered generalized functions. We will refer to both (tempered) generalized functions and (tempered) generalized measures as (tempered) distributions in this paper; the reader will be able to tell the difference from context.

Suppose $W$ is a real, finite dimensional vector space. If $u$ is a tempered generalized measure on $iW^*$, define the Fourier transform of $u$ to be 
$$\mathcal{F}[u](\xi):=\langle u_x,e^{\langle x,\xi\rangle }\rangle,$$
a tempered generalized function on $W$. Further, if $v$ is a tempered generalized function on $W$, define the Fourier transform of $v$ to be $\mathcal{F}[v]=u$ where $u$ is the unique tempered generalized measure on $iW^*$ whose Fourier transform is $v$. In what follows, we will often wish to make estimates on $\mathcal{F}[v]$. For this purpose, we will fix an inner product on $W$, and we let $|\cdot|$ denote the corresponding norm on $W$ and $dx$ the corresponding Lebesgue measure on $W$. We will utilize these, together with division by $i$ to identify $\mathcal{F}[v]$ with a generalized function on $W$ in order to make estimates.

We say a subset $\mathcal{C}$ of a finite-dimensional vector space $W$ is a \emph{cone} if $tv\in W$ whenever $v\in W$ and $t>0$ is a positive real number. If $f$ is a smooth function on a real vector space $W$ and $\mathcal{C}\subset W$ is an open cone, then we say $f$ is \emph{rapidly decaying} in $\mathcal{C}$ if for every $N\in \mathbb{N}$ there exists a constant $C_N>0$ such that
$$|f(x)|\leq  C_N|x|^{-N}$$
for all $x\in \mathcal{C}$. Colloquially, $f$ is rapidly decaying in $\mathcal{C}$ if it decays faster than any rational function in $\mathcal{C}$.

The definition of the (smooth) wave front set of a distribution was first given by H\"{o}rmander on page 120 of \cite{Hor1}. Here we give the most elementary definition (see pages 251-270 of \cite{Hor} for the standard exposition). 

\begin{definition}\label{wfdef} Suppose $u$ is a generalized function on an open subset $X\subset W$, and suppose $(x,\xi)\in X\times iW^*\cong iT^*X$ is a point in the cotangent bundle of $X$. The point $(x,\xi)$ is not in the \emph{wave front set} of $u$ if, and only if there exists an open cone $\xi\in \mathcal{C}\subset iW^*$ and a smooth compactly supported function $\varphi\in C_c^{\infty}(X)$ with $\varphi(x)\neq 0$ such that $\mathcal{F}[\varphi u]$ is rapidly decaying in $\mathcal{C}$. The wave front set of $u$ is denoted $\operatorname{WF}(u)$.
\end{definition}

Many authors use the convention that $(x,0)$ is never in the wave front set for any $x\in X$. However, we will use the convention that the zero section of $iT^*X$ is always in the wave front set because it will make the statements of our results cleaner.

There are several (equivalent) variants of this definition that we will sometimes use. First, instead of a cone $\xi \in\mathcal{C}\subset iW^*$, one may take an open subset $\xi\in \Omega\subset iW^*$ and require $$\mathcal{F}[\varphi u](t\eta)$$ to be rapidly decaying in the variable $t$ for $t>0$ uniformly in the parameter $\eta\in \Omega$. Second, suppose $U\subset X$ is an open set and $\mathcal{C}_1\subset iW^*$ is a closed cone. Then $(U\times \mathcal{C}_1)\cap \operatorname{WF}(u)=U\times \{0\}$ iff for every $\varphi\in C_c^{\infty}(U)$ and every compact subset $0\notin K\subset iW^*-\mathcal{C}_1$, the expression $\mathcal{F}[\varphi u](t\eta)$ is rapidly decaying in $t$ for $t>0$ uniformly for $\eta\in K$ (see page 262 of \cite{Hor}).
Third, instead of a smooth, compactly supported function $\varphi$, one may take an even Schwartz function $\varphi$ that does not vanish at zero and form the family of Schwartz functions
$$\varphi_t(y)=t^{n/4}\varphi(t^{1/2}(y-x))$$
for $t>0$. Then $(x,\xi)$ is not in the wave front set of $u$ iff there exists an open subset $\xi\in \Omega\subset iW^*$ such that
$\mathcal{F}[\varphi_t u](t\eta)$
is rapidly decaying in the variable $t$ for $t>0$ uniformly in $\eta\in \Omega$. This third variant is nontrivial. It is due to Folland (see page 155 of \cite{Fo}); the case where $\varphi$ is a Gaussian was obtained earlier by Cordoba and Fefferman \cite{CF}.

Now, if $\psi\colon X\rightarrow Y$ is a diffeomorphism between two open sets in $W$ and $u$ is a distribution on $Y$, then (see page 263 of \cite{Hor})
$$\psi^*\operatorname{WF}(u)=\operatorname{WF}(\psi^*u).$$
One sees immediately from this functoriality property that the notion of the wave front set of a distribution on a smooth manifold is independent of the choice of local coordinates and is therefore well defined.

We note that the original definition of the wave front set involved pseudodifferential operators instead of abelian harmonic analysis. See page 89 of \cite{Hor2} for a proof that the original definition and the one above are equivalent.
 
\bigskip

The notion of the singular spectrum of a hyperfunction was first introduced by Sato in \cite{Sa}, \cite{KKS}. It was originally called the singular support; however, there is already a standard notion of singular support in the theory of distributions. Therefore, we use the term singular spectrum, which is now widely used. The book \cite{Mo} is a readable introduction to Sato's work. 

Years after Sato's work, Bros and Iagolnitzer introduced the notion of the essential support of a hyperfunction \cite{Ia}. Their definition was subsequently shown to be equivalent to Sato's \cite{Bo}. In his book \cite{Hor}, H\"{o}rmander introduced the notion of the analytic wave front set of a hyperfunction, and he proved that his notion is equivalent to the essential support of Bros and Iagolnitzer. 

We say that a smooth function $f$ on $\mathbb{R}$ is \emph{exponentially decaying} for $t>0$ if there exist constants $\epsilon>0$ and $C>0$ such that
$$|f(t)|\leq Ce^{-\epsilon t}$$
\noindent for $t>0$. We define a family of Gaussians on $\mathbb{R}$ by $$\mathcal{G}_t(s)=e^{-ts^2}.$$
We first give a definition of the singular spectrum that is a variant of the one given by Bros and Iagolnitzer for the essential support.

\begin{definition} \label{ssdefbi} Suppose $u$ is a distribution on an open subset $X\subset W$, and suppose $(x,\xi)\in X\times iW^*\cong iT^*X$ is a point in the cotangent bundle of $X$. The point $(x,\xi)$ is not in the \emph{singular spectrum} of $u$ if, and only if for some (equivalently any) smooth function $\varphi\in C_c^{\infty}(X)$ that is real analytic and nonzero in a neighborhood of $x$, there exists an open set $\xi\in \Omega\subset iW^*$ such that
$$\mathcal{F}[\mathcal{G}_t(|x-y|)\varphi(y)u(y)](t\eta)$$ is exponentially decaying in $t$ for $t>0$ uniformly for $\eta\in \Omega$. The singular spectrum of $u$ is denoted $\operatorname{SS}(u)$.
\end{definition}

In fact, one can extend this definition to hyperfunctions (see Chapter 9 of \cite{Hor}), but we will not need to consider hyperfunctions in this paper. In passing, we note that if $u$ happens to be a tempered distribution, then one need not multiply by the smooth compactly supported function $\varphi$ in the above definition. The nice thing about the above definition is that it is a clear analytic analogue of the Cordoba-Feffermann definition of the smooth wave front set. One simply replaces rapid decay by exponential decay in the definition. However, exponential decay can sometimes be inconvenient to check in some situations. Because of this, we now give an alternate definition of H\"{o}rmander.

For this definition, we need a remark. Fix a basis $\{v_1,\ldots,v_n\}$ of the finite dimensional, real vector space $W$. Suppose $U_1\subset U\subset W$ are precompact open sets with $U_1$ compactly contained in $U$. For every multi-index $\alpha=(\alpha_1,\ldots,\alpha_n)$, define the differential operator $$D^{\alpha}=\partial^{\alpha_1}_{v_1}\cdots \partial^{\alpha_n}_{v_n},$$
and define $|\alpha|:=\alpha_1+\cdots+\alpha_n$. Then there exists (see pages 25-26, 282 of \cite{Hor}) a sequence $\varphi_{N,U_1,U}$ of smooth functions supported in $U$ together with a family of positive constants $\{C_{\alpha}\}$ for every multi-index $\alpha=(\alpha_1,\ldots,\alpha_n)$ such that $\varphi_{N,U_1,U}(y)=1$ whenever $y\in U_1$ and 
\begin{equation}\label{eq:varphi_family_bound}
\sup_{y\in U}|D^{\alpha+\beta}\varphi_{N,U_1,U}(y)|\leq C_{\alpha}^{|\beta|+1} (N+1)^{|\beta|}
\end{equation}
whenever $|\beta|\leq N$. For each such pair of precompact open subsets $U_1\subset U\subset W$, we fix such a sequence $\varphi_{N,U_1,U}$.

We now give a variant of H\"{o}rmander's definition of the analytic wave front set of a distribution (see pages 282-283 of \cite{Hor}).

\begin{definition} \label{ssdefh} Suppose $u$ is a distribution on an open set $X\subset W$, and suppose $(x,\xi)\in X\times iW^*\cong iT^*X$ is a point in the cotangent bundle of $X$. The point $(x,\xi)$ is not in the \emph{singular spectrum} of $u$ if, and only if there exists a pair of precompact open sets $x\in U_1\subset U\subset X$ with $U_1$ compactly contained in $U$, an open set $\xi\in \Omega\subset iW^*$, and a constant $C>0$ such that for every $N\in \mathbb{N}$, we have the estimate
$$|\mathcal{F}[\varphi_{N,U_1,U} u](t\eta)|\leq C^{N+1}(N+1)^N t^{-N}$$
uniformly for $\eta\in \Omega$. The singular spectrum of $u$ is denoted $\operatorname{SS}(u)$.
\end{definition}

One key disadvantage of the definitions of Bros-Iagolnitzer and H\"{o}rmander is that they are not obviously invariant under analytic changes of coordinates. This is certainly an advantage of the original definition of Sato. However, in this paper, we will use the close relationship between the analytic wave front set of a distribution and the ability to write the distribution as the boundary value of a complex analytic function. This relationship is originally due to Sato \cite{KKS}, \cite{Mo}; however, we will follow the treatment in Sections 8.4, 8.5 of \cite{Hor}. We will use this theory in Section \ref{sec:regular_part2}. For now, we remark on the following application.

If $\psi\colon X\rightarrow Y$ is a bianalytic isomorphism between two open sets in $W$ and $u$ is a distribution on $Y$, then (see page 296 of \cite{Hor})
$$\psi^*\operatorname{SS}(u)=\operatorname{SS}(\psi^*u).$$
One sees immediately from this functoriality property that the notion of the singular spectrum of a distribution on an analytic manifold is independent of the choice of analytic local coordinates and is therefore well defined.

Finally, we remark that if $u$ is a distribution on an analytic manifold, then we have 
$$\operatorname{SS}(u)\supset \operatorname{WF}(u).$$
This is obvious from the above definitions. Roughly speaking, it means that it is tougher for $u$ to be analytic than smooth.

\bigskip

Suppose $G$ is a Lie group, $(\pi,V)$ is a unitary representation of $G$, and $(\cdot,\cdot)$ is the inner product on the Hilbert space $V$. As in the introduction, we define the \emph{wave front set of} $\pi$ and the \emph{singular spectrum of} $\pi$ by
$$\operatorname{WF}(\pi)=\overline{\bigcup_{u,v\in V}\operatorname{WF}_e(\pi(g)u,v)},\ \  \operatorname{SS}(\pi)=\overline{\bigcup_{u,v\in V}\operatorname{SS}_e(\pi(g)u,v)}.$$
Here the subscript $e$ means that we are only taking the piece of the wave front set (or singular spectrum) in the fiber over the identity in $iT^*G$. One might ask why we add this restriction. Utilizing the short argument on page 118 of \cite{How}, one observes that 
$$\overline{\bigcup_{u,v\in V}\operatorname{WF}(\pi(g)u,v)},\ \ \overline{\bigcup_{u,v\in V}\operatorname{SS}(\pi(g)u,v)}$$
are $G\times G$ invariant, closed subsets of $iT^*G\cong G\times i\mathfrak{g}^*$. In particular, they are simply $G\times \operatorname{WF}(\pi)$ and $G\times \operatorname{SS}(\pi)$. Therefore, if we did not add the the subscript $e$ in our definitions of the wave front set and singular spectrum of $\pi$, then we would simply be taking the product of our sets with $G$. This would be more cumbersome and no more enlightening. 

We note in passing that the above digression together with the above definitions of the wave front set and singular spectrum of a distribution imply that $\operatorname{WF}(\pi)$ and $\operatorname{SS}(\pi)$ are closed, $\operatorname{Ad}^*(G)$-invariant cones in $i\mathfrak{g}^*$. We also note that
$$\operatorname{SS}(\pi)\supset \operatorname{WF}(\pi)$$ for every unitary Lie group representation $\pi$ since $\operatorname{SS}_e(u)\supset \operatorname{WF}_e(u)$ whenever $u$ is a distribution on an analytic manifold.

Let $\mathcal{B}^1(V)$ denote the Banach space of trace class operators on $V$. Given a trace class operator $T\in \operatorname{End}V$, one can define a continuous function on $G$ by
$$\operatorname{Tr}_{\pi}(T)(g)=\operatorname{Tr}(\pi(g)T).$$

We define
$$\widetilde{\operatorname{WF}(\pi)}=\overline{\bigcup_{T\in \mathcal{B}^1(V)} \operatorname{WF}_e(\operatorname{Tr}_{\pi}(T)(g))},\ \widetilde{\operatorname{SS}(\pi)}=\overline{\bigcup_{T\in \mathcal{B}^1(V)} \operatorname{SS}_e(\operatorname{Tr}_{\pi}(T)(g))}.$$

The definition on the left was $i$ times the original definition used by Howe for $\operatorname{WF}^0(\pi)$ \cite{How}. Notice that when $T=(\cdot,u)v$ is a rank one operator, $\operatorname{Tr}_{\pi}(T)(g)=(\pi(g)u,v)$ is a matrix coefficient. Therefore, it is clear from our definitions that $\operatorname{WF}(\pi)\subset \widetilde{\operatorname{WF}(\pi)}$ and $\operatorname{SS}(\pi)\subset \widetilde{\operatorname{SS}(\pi)}$. The primary purpose of the remainder of this section is to prove equality.

\begin{proposition}\label{howeusequiv} We have 
$$\operatorname{WF}(\pi)=\widetilde{\operatorname{WF}(\pi)}\ \text{and}\ \operatorname{SS}(\pi)=\widetilde{\operatorname{SS}(\pi)}.$$
\end{proposition}

To prove the Proposition, we will need to recall some facts about wave front sets of representations from \cite{How}. 
If $T\in \operatorname{End}V$ is a bounded linear operator, let $|T|_{\infty}$ denote the operator norm of $T$. If $T\in \mathcal{B}^1(V)$ is a trace class operator, let $|T|_1$ denote the trace class norm of $T$. 

\begin{lemma} [Howe] \label{howelemmawf} Suppose $G$ is a Lie group, and $(\pi,V)$ is a unitary representation of $G$. The following are equivalent:\\
\begin{enumerate}[(a)]
\item $\xi\notin \widetilde{\operatorname{WF}(\pi)}$
\item For every $T\in \mathcal{B}^1(V)$, there exists an open set $e\in U\subset G$ on which the logarithm is a well-defined diffeomorphism onto its image and an open set $\xi\in \Omega\subset i\mathfrak{g}^*$ such that for every $\varphi\in C_c^{\infty}(U)$, the absolute value of the integral
$$I(\varphi,\eta,T)(t)=\int_G \operatorname{Tr}_{\pi}(T)(g) e^{t\eta(\log g)}\varphi(g)dg$$
is rapidly decaying in $t$ for $t>0$ uniformly for $\eta\in \Omega$.
\item There exists an open set $e\in U\subset G$ on which the logarithm is a well-defined diffeomorphism onto its image and an open set $\xi\in \Omega\subset i\mathfrak{g}^*$ such that for every $\varphi\in C_c^{\infty}(U)$ there exists a family of constants $C_N(\varphi)>0$ such that 
$$|I(\varphi,\eta,T)(t)|\leq C(\varphi)|T|_1t^{-N}$$
for $t>0$, $\eta\in \Omega$, and $T\in \mathcal{B}^1(V)$. (The constants $C(\varphi)$ may be chosen independent of both $\eta\in W$ and $T\in \mathcal{B}^1(V)$).
\item There exists an open set $e\in U\subset G$ on which the logarithm is a well-defined diffeomorphism onto its image and an open set $\xi\in \Omega\subset i\mathfrak{g}^*$ such that for every $\varphi\in C_c^{\infty}(U)$, the quantity
$$|\pi(\varphi(g) e^{t\eta(\log g)})|_{\infty}$$
is rapidly decaying in $t$ for $t>0$ uniformly in $\eta\in \Omega$.
\end{enumerate}
\end{lemma}

This Lemma is a subset of Theorem 1.4 of \cite{How}. Some of the notation has been slightly altered for convenience. Next, we need an analogue of this Lemma for our first definition of the singular spectrum, Definition \ref{ssdefbi}.

\begin{lemma} \label{howelemmassbi} Suppose $G$ is a Lie group and $(\pi,V)$ is a unitary representation of $G$. The following are equivalent:\\
\begin{enumerate}[(a)]
\item $\xi\notin \widetilde{\operatorname{SS}(\pi)}$
\item For every $T\in \mathcal{B}^1(V)$ and for some (equivalently every) pair of precompact open sets $e\in U_1\subset U\subset G$ with $U_1$ compactly contained in $U$ and so that the logarithm on $U$ is a well-defined bianalytic isomorphism onto its image, there exists an open set $\xi\in \Omega\subset i\mathfrak{g}^*$ such that for some (equivalently every) $\varphi\in C_c^{\infty}(U)$ that is identically one on $U_1$, the absolute value of the integral
$$I(\varphi,\eta,T)(t)=\int_G \operatorname{Tr}_{\pi}(T)(g) e^{t\eta(\log g)}\varphi(g)\mathcal{G}_t(|\log(g)|)dg$$
is exponentially decaying in $t$ for $t>0$ uniformly for $\eta\in \Omega$.
\item For some (equivalently every) pair of precompact open sets $e\in U_1\subset U\subset G$ with $U_1$ compactly contained in $U$ and so that the logarithm on $U$ is a well-defined bianalytic isomorphism onto its image, there exists an open set $\xi\in \Omega\subset i\mathfrak{g}^*$ such that for some (equivalently every) $\varphi\in C_c^{\infty}(U)$ that is identically one on $U_1$, there exist constants $C(\varphi)>0$ and $\epsilon(\varphi)>0$ such that 
$$|I(\varphi,\eta,T)(t)|\leq C(\varphi)|T|_1e^{-\epsilon(\varphi) t}$$
for $t>0$, $\eta\in \Omega$, and $T\in \mathcal{B}^1(V)$. (The constants $C(\varphi)$ and $\epsilon(\varphi)$ may be chosen independent of both $\eta\in \Omega$ and $T\in \mathcal{B}^1(V)$).
\item For some (equivalently every) pair of precompact open sets $e\in U_1\subset U\subset G$ with $U_1$ compactly contained in $U$ and so that the logarithm on $U$ is a well-defined bianalytic isomorphism onto its image, there exists an open set $\xi\in \Omega\subset i\mathfrak{g}^*$ such that for some (equivalently every) $\varphi\in C_c^{\infty}(U)$ that is identically one on $U_1$, the quantity
$$|\pi(\varphi(g) \mathcal{G}_t(|\log(g)|) e^{t\eta(\log g)})|_{\infty}$$
is exponentially decaying in $t$ for $t>0$ uniformly in $\eta\in \Omega$.
\end{enumerate}
\end{lemma}

We note that the proof of Lemma \ref{howelemmassbi} is nearly identical to the proof of Lemma \ref{howelemmawf}. As noted before, Lemma \ref{howelemmawf} is part of Theorem 1.4 on page 122 of \cite{How}.

Next, we prove Proposition \ref{howeusequiv}.

\begin{proof} In both cases, one containment is obvious. Therefore, to prove the Lemma it is enough to show
$$\overline{\bigcup_{T\in \mathcal{B}^1(V)} \operatorname{WF}_e(\operatorname{Tr}_{\pi}(T))} \subset \overline{\bigcup_{v,w\in V} \operatorname{WF}_e(\pi(g)v,w)}$$
and 
$$\overline{\bigcup_{T\in \mathcal{B}^1(V)} \operatorname{SS}_e(\operatorname{Tr}_{\pi}(T))} \subset \overline{\bigcup_{v,w\in V} \operatorname{SS}_e(\pi(g)v,w)}.$$

In particular, it is enough to fix $$\xi\notin \overline{\bigcup_{v,w\in V} \operatorname{WF}_e(\pi(g)v,w)},\ \zeta\notin \overline{\bigcup_{v,w\in V} \operatorname{SS}_e(\pi(g)v,w)}$$ and then show that 
$$\xi\notin \overline{\bigcup_{T\in \mathcal{B}^1(V)} \operatorname{WF}_e(\operatorname{Tr}_{\pi}(T))},\ \zeta\notin \overline{\bigcup_{T\in \mathcal{B}^1(V)} \operatorname{SS}_e(\operatorname{Tr}_{\pi}(T))}.$$

By the second variant of Definition \ref{wfdef}, we may find an open neighborhood $e\in U\subset G$ on which the logarithm is well-defined and an open neighborhood $\xi\in \Omega\subset i\mathfrak{g}^*$ such that for all $N\in \mathbb{N}$ and $\varphi\in C_c^{\infty}(U)$ the quantity
$$\left|t^N\int_{U} \varphi(g)e^{t\langle \log(g),\eta\rangle}(\pi(g)v,w)dg\right|$$
is bounded as a function of $\eta\in \Omega$ and $t>0$ for every $v,w\in V$. By the uniform boundedness principle, we deduce that the family of operators $t^N\pi(\varphi(g)e^{it\langle \log(g),\eta\rangle})$ is uniformly bounded in the operator norm for $\eta\in \Omega$ and $t>0$. Therefore
$$\left|\pi(\varphi(g)e^{it\langle \log(g),\eta\rangle})\right|_{\infty}$$
is rapidly decreasing in $t$ for $t>0$ uniformly in $\eta\in \Omega$. Utilizing Lemma \ref{howelemmawf}, the first statement follows.

For the singular spectrum case, by Definition \ref{ssdefbi}, we may find a pair of precompact open neighborhoods $e\in U_1\subset U\subset G$ on which the logarithm is well-defined and an open neighborhood $\zeta\in \Omega\subset i\mathfrak{g}^*$ such that for some $\varphi\in C_c^{\infty}(U)$ with $\varphi=1$ on $U_1$, we have
$$\left|\int_{U} \varphi(g)\mathcal{G}_t(|\log(g)|)e^{t\langle \log(g),\eta\rangle}(\pi(g)v,w)dg\right|\leq C_{v,w}(\varphi)e^{-\epsilon(v,w,\varphi)t}$$
for $t>0$ and $\eta\in \Omega$. We must show that the above constants $C_{v,w}(\varphi)$ and $\epsilon(v,w,\varphi)$ are independent of $v$ and $w$ subject to the conditions $|v|=|w|=1$. Denote the above integral by $I(\varphi,\eta,v,w)(t)$ and fix $v$. Let 
$$S_n(v)=\{w\in V|\ |I(\varphi,\eta,v,w)(t)|\leq ne^{-(1/n)t}\ \text{uniformly\ for}\ \eta\in \Omega\}.$$
By the Baire Category Theorem and the linearity of $I$ in the variable $w$, we observe that $S_{n_v}(v)$ contains a $\delta$ ball, $B_{\delta}(0)$, around zero for some $n_v$. In particular, for fixed $v$, the constants $C_{v,w}(\varphi)$ and $\epsilon(v,w,\varphi)$ can be taken independent of $w$ with $|w|=1$ ($C_{v,w}(\varphi)=n_v/\delta$, $\epsilon(v,w,\varphi)=1/n_v$ in the above argument).

In particular, we may find a pair of precompact open neighborhoods $e\in U_1\subset U\subset G$ on which the logarithm is well-defined and an open neighborhood $\zeta\in \Omega\subset i\mathfrak{g}^*$ such that for some $\varphi\in C_c^{\infty}(U)$ with $\varphi=1$ on $U_1$, we have
$$\left|\int_{U} \varphi(g)\mathcal{G}_t(|\log(g)|)e^{t\langle \log(g),\eta\rangle}\pi(g)v dg\right|\leq C_v(\varphi)e^{-\epsilon(v,\varphi)t}$$
for $t>0$ and $\eta\in \Omega$. Denote the integral on the left by $I(\varphi,\eta,v)(t)$ and set
$$S_n=\{v\in V|\ |I(\varphi,\eta,v)(t)|\leq ne^{-(1/n)t}\ \text{uniformly\ for}\ \eta\in \Omega\}.$$
Utilizing the Baire Category Theorem and the linearity of $I(\varphi,\eta,v)$ in the variable $v$, we observe that there exists $N$ for which $S_N$ contains a $\delta$ ball, $B_{\delta}(0)$, about the origin. In particular, we may set $C_v(\varphi)=N/\delta$ and $\epsilon(v,\varphi)=1/N$ in the above inequality for all $v\in V$ with $|v|=1$. It follows that 
$$|\pi(\varphi(g)\mathcal{G}_t(|\log(g)|)e^{t\langle \log(g),\eta\rangle})|_{\infty}$$
is exponentially decaying in $t$ for $t>0$ uniformly for $\eta\in \Omega$. The second statement in Proposition \ref{howeusequiv} now follows from Lemma \ref{howelemmassbi}.
\end{proof}

\section{Wave Front Sets and Distribution Vectors}
\label{sec:dist_vec}

If $(\pi,V)$ is a unitary representation of a Lie group $G$, then 
$$V^{\infty}=\{v\in V| g\mapsto \pi(g)v\ \text{is\ smooth}\}.$$
The Lie algebra $\mathfrak{g}$ acts on $V^{\infty}$, and we give $V^{\infty}$ a complete, locally convex topology via the seminorms $|v|_D=|Dv|$ for each $D\in \mathcal{U}(\mathfrak{g})$. Now, given a unitary representation $(\pi,V)$, we may form the conjugate representation $(\overline{\pi},\overline{V})$ by simply giving $V$ the conjugate complex structure. Define $V^{-\infty}$ to be the dual space of $\overline{V}^{\infty}$.

Given $\zeta,\eta\in V^{-\infty}$, we wish to define a generalized matrix coefficient denoted by $(\pi(g)\zeta,\eta)$. This generalized matrix coefficient will be a generalized function on $G$. To define it, we need a couple of preliminaries.
Suppose $\mu\in C_c^{\infty}(G,\mathcal{D}(G))$ is a smooth, compactly supported section of the complex density bundle $\mathcal{D}(G)\rightarrow G$ on $G$, and suppose $\zeta\in V^{-\infty}$. Then we define
$\pi(\mu)\zeta\in V^{-\infty}$ by
$$\langle \pi(\mu)\zeta,\overline{v}\rangle=\langle \zeta,\overline{\pi}(\iota^*\mu)\overline{v}\rangle=\langle \zeta,\int_G \overline{\pi}(g)\overline{v} d\mu(g^{-1})\rangle$$
for $\overline{v}\in \overline{V}^{\infty}$. Here $\iota$ denotes inversion on the group $G$.

\begin{lemma} \label{distvec} For $\mu\in C_c^{\infty}(G,\mathcal{D}(G))$ and $\zeta\in V^{-\infty}$, we have $\pi(\mu)\zeta\in V^{\infty}$. Moreover, if $\zeta,\eta\in V^{-\infty}$, then the linear functional
$$\mu \mapsto (\pi(\mu)\zeta,\eta)$$
is continuous and therefore defines a distribution on $G$. We will denote this distribution by
$(\pi(g)\zeta,\eta)$.
\end{lemma}

This Lemma has been well-known to experts for some time. For a proof, see the exposition on pages 9-13 of \cite{he}.

In fact, we may define the (smooth or analytic) wave front set of a unitary representation in terms of the (smooth or analytic) wave front sets of the generalized matrix coefficients of $G$.

\begin{proposition} \label{wfdistvec} We have the equalities
$$\operatorname{WF}(\pi)=\overline{\bigcup_{\zeta, \eta\in V^{-\infty}} \operatorname{WF}_e(\pi(g)\zeta,\eta)}$$
and 
$$\operatorname{SS}(\pi)=\overline{\bigcup_{\zeta, \eta\in V^{-\infty}} \operatorname{SS}_e(\pi(g)\zeta,\eta)}.$$
\end{proposition}

The key to this Proposition is the following Lemma.

\begin{lemma} \label{diffvect} If $\zeta\in V^{-\infty}$, then there exists $D\in \mathcal{U}(\mathfrak{g})$ and $u\in V$ such that $Du=\zeta$.
\end{lemma}
 
This Lemma has been well-known to experts for some time. For a proof, see the exposition on page 5 of \cite{he}.

Now, we prove the Proposition.
\begin{proof} Clearly the left hand sides are contained in the right hand sides. To show the other directions, fix $\zeta,\eta\in V^{-\infty}$. Write $\zeta=D_1u$ and $\eta=D_2v$ with $D_1,D_2\in \mathcal{U}(\mathfrak{g})$ and $u,v\in V$. Then we have
$$\operatorname{WF}(\pi(g)\zeta,\eta)=\operatorname{WF}(L_{D_2}R_{D_1}(\pi(g)u,v))$$
and 
$$\operatorname{SS}(\pi(g)\zeta,\eta)=\operatorname{SS}(L_{D_2}R_{D_1}(\pi(g)u,v)).$$
Here $R_{D_1}$ (resp. $L_{D_2}$) denotes the action of $D_1$ (resp. $D_2$) via right (resp. left) translation on $C^{-\infty}(G)$. But, by (8.1.11) on page 256 of \cite{Hor}, we deduce
$$\operatorname{WF}(L_{D_2}R_{D_1}(\pi(g)u,v))\subset \operatorname{WF}(\pi(g)u,v).$$
And from the remark on the top of page 285 of \cite{Hor}, we deduce
$$\operatorname{SS}(L_{D_2}R_{D_1}(\pi(g)u,v))\subset \operatorname{SS}(\pi(g)u,v).$$
The Proposition follows.
\end{proof}

\section{Wave Front Sets of Induced Representations}
\label{sec:wf_induced}

Now, suppose $H\subset G$ is a closed subgroup, and let $\mathcal{D}^{1/2}\rightarrow G/H$ be the bundle of complex half densities on $G/H$. Let $(\tau,W)$ be a unitary representation of $H$, and let $\mathcal{W}\rightarrow G/H$ be the corresponding invariant, Hermitian (possibly infinite-dimensional) vector bundle on $G/H$. Then we obtain a unitary representation of $G$ by letting $G$ act by left translation on
$$L^2(G/H, \mathcal{W}\otimes\mathcal{D}^{1/2}).$$ This representation is usually denoted by $\operatorname{Ind}_H^G\tau$; it is called the representation of $G$ induced from the representation $\tau$ of $H$ (sometimes the term ``unitarily induced'' is used). Let $\mathfrak{g}$ (resp. $\mathfrak{h}$) denote the Lie algebra of $G$ (resp. $H$), and let $q\colon i\mathfrak{g}^*\rightarrow i\mathfrak{h}^*$ be the pullback of the inclusion. If $S\subset i\mathfrak{h}^*$, we define
$$\operatorname{Ind}_H^G S=\overline{\operatorname{Ad}^*(G)\cdot q^{-1}(S)}.$$
If $S$ is a cone, then $\operatorname{Ind}_H^G S$ is the smallest closed, $\operatorname{Ad}^*(G)$ invariant cone in $i\mathfrak{g}^*$ that contains $q^{-1}(S)$. The purpose of this section is to prove Theorem \ref{inducedintro}. Recall that we must show $$\operatorname{WF}(\operatorname{Ind}_H^G\tau)\supset \operatorname{Ind}_H^G\operatorname{WF}(\tau)$$
and 
$$\operatorname{SS}(\operatorname{Ind}_H^G\tau)\supset \operatorname{Ind}_H^G\operatorname{SS}(\tau).$$

We note that $\operatorname{WF}(\operatorname{Ind}_H^G\tau)$ and $\operatorname{SS}(\operatorname{Ind}_H^G\tau)$ are closed, $\operatorname{Ad}^*(G)$ invariant cones in $i\mathfrak{g}^*$. Therefore, to show that $\operatorname{WF}(\operatorname{Ind}_H^G\tau)$ contains $\operatorname{Ind}_H^G\operatorname{WF}(\tau)$ (respectively $\operatorname{SS}(\operatorname{Ind}_H^G\tau)$ contains $\operatorname{Ind}_H^G\operatorname{SS}(\tau)$), it is enough to show that $\operatorname{WF}(\operatorname{Ind}_H^G\tau)$ contains $q^{-1}(\operatorname{WF}(\tau))$ (respectively $\operatorname{SS}(\operatorname{Ind}_H^G\tau)$ contains $q^{-1}(\operatorname{SS}(\tau))$).

Before proving the Theorem, we first make a few general comments and then we will prove a Lemma. Suppose $H\subset G$ is a closed subgroup of a Lie group. Let $\mathcal{D}(G)\rightarrow G$ (resp. $\mathcal{D}(H)\rightarrow H$, $\mathcal{D}(G/H)\rightarrow G/H$) denotes the complex density bundle on $G$ (resp. $H$, $G/H$). Now, suppose we are given $f\in C(H)$, a continuous function on $H$, and $\omega\in \mathcal{D}_H(G/H)^*$, an element of the dual of the fiber over $\{H\}$ in the density bundle on $G/H$. We claim that $f\omega$ defines a generalized function on $G$.\\
\indent To see this, we must show how to pair $f\omega$ with a smooth, compactly supported density, $\mu$, on $G$. Let $n=\dim G$, $m=\dim H$, and recall that for each $h\in H$, $\mu_h$ is a map
$$\mu_h\colon \mathfrak{g}^{\oplus n}\rightarrow \mathbb{C}$$
satisfying $$\mu_h(AX_1,\ldots,AX_{n})=|\det A|\mu_h(X_1,\ldots,X_{n})$$
for $A\in \operatorname{End}(\mathfrak{g})$ and $X_1,\ldots,X_{n}\in \mathfrak{g}$. 
Similarly, $\omega$ is a map
$$\omega\colon (\mathfrak{g}/\mathfrak{h})^{\oplus (n-m)}\rightarrow \mathbb{C}$$
satisfying $$\omega(AX_1,\ldots,AX_{n-m})=|\det A|^{-1}\omega(X_1,\ldots,X_{n-m})$$
for $A\in \operatorname{Aut}(\mathfrak{g}/\mathfrak{h})$ and $X_1,\ldots,X_{n-m}\in \mathfrak{g}/\mathfrak{h}$.\\
\indent To pair $f\omega$ with $\mu$, we must show that $\mu \omega$ defines a smooth, compactly supported density on $H$. For each $h\in H$, we will define a map 
$$\mu_h \omega\colon \mathfrak{h}^{\oplus m}\rightarrow \mathbb{C}.$$
To do this, we fix $Y_1,\ldots,Y_{n-m}\in \mathfrak{g}$ such that $\{\overline{Y_1},\ldots,\overline{Y}_{n-m}\}$ is a basis for $\mathfrak{g}/\mathfrak{h}$. Then we define
$$(\mu_h\omega)(X_1,\ldots,X_{m})=$$
$$\mu_h(X_1,\ldots,X_{m},Y_1,\ldots,Y_{n-m})\omega(\overline{Y_1},\ldots,\overline{Y}_{n-m})$$
for any $X_1,\ldots,X_{m}\in \mathfrak{h}$.  One checks directly that this definition of $\mu_h\omega$ is independent of the choice of $Y_1,\ldots,Y_{n-m}$ and that it satisfies
$$(\mu_h\omega)(AX_1,\ldots,AX_{m})=|\det A|(\mu_h\omega)(X_1,\ldots,X_{m})$$
for $A\in \operatorname{End}(\mathfrak{h})$ and $X_1,\ldots,X_{m}\in \mathfrak{h}$. In particular, $\mu\omega$ is a smooth, compactly supported density on $H$, and the pairing
$$\langle f\omega,\mu\rangle=\langle f,\mu\omega\rangle$$
is well-defined and continuous. Thus, $f\omega$ defines a generalized function on $G$.
\bigskip

Now, recall $(\tau,W)$ is a unitary representation of $H$. For $w_1\in W$ and a non-zero $\omega_1\in (\mathcal{D}_{H}(G/H)^{1/2})^*$, we define a distribution vector $$\delta_H(w_1,\omega_1)\in C_c^{-\infty}(G/H,\mathcal{W}\otimes \mathcal{D}^{1/2})\cong C^{\infty}(G/H,\overline{\mathcal{W}}\otimes \mathcal{D}^{1/2})^*$$ by
$$\delta_H(w_1,\omega_1)\colon \varphi\mapsto \langle \varphi(H),w_1\otimes \omega_1\rangle.$$
The above pairing is the tensor product of the pairing between $W$ and $\overline{W}$ via the inner product on the Hilbert space $W$ and the pairing between $\mathcal{D}_H(G/H)^{1/2}$ and its dual. Similarly, if $w_2\in W$ and $\omega_2\in (\mathcal{D}_{H}(G/H)^{1/2})^*$ is non-zero, we define a distribution vector $$\delta_H(w_2,\omega_2)\in C_c^{-\infty}(G/H,\mathcal{W}\otimes \mathcal{D}^{1/2})\cong C^{\infty}(G/H,\overline{\mathcal{W}}\otimes \mathcal{D}^{1/2})^*.$$
Now, we have a continuous inclusion
$$L^2(G/H,\overline{\mathcal{W}}\otimes \mathcal{D}^{1/2})^{\infty}\subset C^{\infty}(G/H,\overline{\mathcal{W}}\otimes \mathcal{D}^{1/2}).$$
To deduce continuity, we must recall the local Sobolev inequalities. Such an equality for complex valued functions on $\mathbb{R}^n$ can be found in the last displayed equation on page 124 of \cite{FJ} where we assume $s=m$ is a positive integer and we take the version of the Sobolev norm given in Theorem 9.3.3 on page 122 of \cite{FJ}. Although a global inequality is given on $\mathbb{R}^n$ in \cite{FJ}, one can derive from it a local inequality which can be transferred to any smooth manifold. In particular, if $f$ is a complex valued, smooth function on a manifold $X$ with a fixed, smooth density $\omega$, and $U\subset X$ is a precompact, open subset, we obtain
$$\sup_{x\in U} |f(x)|\leq C \sum_{j=1}^N |D_jf|_{L^2(X,\omega)}$$
for some constant $C$ and some finite collection of differential operators $\{D_j\}$, both of which may depend on $U$. One can derive the identical inequality for a function $f$ valued in a separable Hilbert space by choosing an orthonormal basis $\{e_j\}$ and deducing the equality for $f$ from the equalities for the complex valued functions $(f,e_j)$. Locally, the density bundle $\mathcal{D}^{1/2}$ is trivial; hence, we may use this inequality to deduce that the above inclusion is continuous. 
Dualizing, we obtain a continuous inclusion
$$C^{-\infty}_c(G/H,\mathcal{W}\otimes \mathcal{D}^{1/2}) \subset L^2(G/H,\mathcal{W}\otimes \mathcal{D}^{1/2})^{-\infty}.$$
Therefore, since the distributions $\delta_{H}(w_1,\omega_1)$ and $\delta_H(w_2,\omega_2)$ are supported at a single point, they are compactly supported and by the above inclusion they both define distribution vectors for the representation $L^2(G/H,\mathcal{W}\otimes \mathcal{D}^{1/2})$.

\begin{lemma} \label{identityvector} The distribution on $G$ defined by the generalized matrix coefficient $$(\pi(g)\delta_H(w_1,\omega_1),\delta_H(w_2,\omega_2))$$ (see Lemma \ref{distvec}) is equal to the generalized function on $G$ defined by
$$\mu\mapsto |\det(\operatorname{Ad}(h)|_{\mathfrak{g}/\mathfrak{h}})|\cdot ((\tau(h)w_1,w_2)\omega, \mu)$$
where $\mu$ is a smooth, compactly supported section of the density bundle on $G$ and $\omega=\omega_1\omega_2\in \mathcal{D}_H(G/H)^*$.
\end{lemma}

\begin{proof} We will prove the Lemma by directly analyzing the generalized matrix coefficient $(\pi(g)\delta_H(w_1,\omega_1),\delta_H(w_2,\omega_2))$. Fix $\mu\in C_c^{\infty}(G,\mathcal{D}(G))$ a smooth, compactly supported density on $G$. By Lemma \ref{distvec},
$$\pi(\mu)\delta_H(w_1,\omega_1)$$ is a smooth vector in 
$$L^2(G/H,\mathcal{W}\otimes \mathcal{D}^{1/2})^{\infty}\subset C^{\infty}(G/H,\mathcal{W}\otimes \mathcal{D}^{1/2}).$$ Pairing it with $\delta_H(w_2,\omega_2)$ means evaluating this smooth function at $\{H\}$ and pairing it with $w_2\otimes \omega_2$. First, we wish to analyze the smooth function 
$\pi(\mu)\delta_H(w_1,\omega_1)$ by pairing it with $\psi\in C_c^{\infty}(G/H,\overline{\mathcal{W}}\otimes \mathcal{D}^{1/2})$. We have
$$\langle \pi(\mu)\delta_H(w_1,\omega_1), \psi\rangle=\int_G (w_1\otimes \omega_1,L_{g^{-1}}\psi(\overline{g})) d\mu(g).$$
Now, $\omega\in \mathcal{D}_H(G/H)^*$ is a vector in the dual of fiber of the density bundle on $G/H$ above $H$. We let $\omega^*\in \mathcal{D}_H(G/H)$ be the unique vector so that $\langle \omega^*,\omega\rangle=1$. Moreover, extend $\omega^*$ to a nonvanishing section $\widetilde{\omega^*}$ of the complex density bundle on $G/H$. Now, if $\varphi\in C_c^{\infty}(G)$, then instead of integrating $\varphi \mu$ over $G$, we wish to integrate over the fibers of the fibration 
$$G\rightarrow G/H$$
which are simply the cosets $xH$ and then integrate against $\widetilde{\omega^*}$ along the base. One sees that for every $gH\in G/H$, there exists a smooth density $\eta_{gH}\in C^{\infty}(gH, \mathcal{D}(gH))$ such that
$$\int_{G} \varphi(g)\mu(g)=\int_{G/H}\left(\int_H \varphi(gh) d\eta_{gH}(h)\right)d\widetilde{\omega^*}(\overline{g}).$$
In addition, note $\eta_{H}\omega^*=\mu$ and $\eta_{H}=\mu \omega$.
We apply this integration formula for
$$\varphi(g)=(w_1\otimes \omega_1,L_{g^{-1}}\psi(\overline{g})).$$
Thus, we obtain
$$\langle \pi(\mu)\delta_H(w_1,\omega_1), \psi\rangle$$
$$=\int_{G/H}\left[\int_H (w_1\otimes \omega_1,L_{(gh)^{-1}}\psi(\overline{g})) d\eta_{gH}(h)\right]d\widetilde{\omega^*}(\overline{g})$$
$$=\int_{G/H}\left(\int_H L_g(\tau(h)w_1\otimes h\cdot \omega_1) d\eta_{gH}(h),\psi(\overline{g})\right)d\widetilde{\omega^*}(\overline{g}).$$

One sees the distribution $\pi(\mu)\delta_H(w_1,\omega_1)$ is the smooth function with values in the bundle $\mathcal{W}\otimes \mathcal{D}(G/H)^{1/2}$ given by
$$g\mapsto \left(\int_H L_g(\tau(h)w_1\otimes h\cdot \omega_1)d\eta_{gH}(h)\right)\cdot \widetilde{\omega^*}.$$
Evaluating at $\{H\}$ yields 
$$\left(\int_H \tau(h)w_1\otimes h\cdot \omega_1d\eta_{H}(h)\right)\cdot \omega^*.$$
Now, $h\cdot \omega_1=|\det(\operatorname{Ad}(h)|_{\mathfrak{g}/\mathfrak{h}})|\cdot \omega_1$. Pairing with $w_2\otimes \omega_2$ yields
$$\int_{H} |\det(\operatorname{Ad}(h)|_{\mathfrak{g}/\mathfrak{h}})|(\tau(h)w_1,w_2)\langle \omega_1\omega^*,\omega_2\rangle\eta_{H}(h)$$
$$=\langle |\det(\operatorname{Ad}(h)|_{\mathfrak{g}/\mathfrak{h}})|\cdot (\tau(h)w_1,w_2), \mu \omega\rangle.$$
Here we have used $\langle \omega_1\omega_2,\omega^*\rangle=1$ and $\eta_{H}=\mu\omega$. The Lemma follows.
\end{proof}

Now, we are ready to prove Theorem \ref{inducedintro}.

\begin{proof} Let $w_1, w_2\in W$ be two vectors, and let $(\tau(h)w_1,w_2)$ be the corresponding matrix coefficient of $(\tau,W)$. To prove the Theorem, it is enough to show
$$\operatorname{WF}(L^2(G/H,\mathcal{W}\otimes \mathcal{D}^{1/2}))\supset q^{-1}(\operatorname{WF}_e(\tau(h)w_1,w_2))$$
and
$$\operatorname{SS}(L^2(G/H,\mathcal{W}\otimes \mathcal{D}^{1/2}))\supset q^{-1}(\operatorname{SS}_e(\tau(h)w_1,w_2)).$$

Let $V=L^2(G/H,\mathcal{W}\otimes \mathcal{D}^{1/2})$ and recall the equalities
$$\operatorname{WF}(\pi)=\overline{\bigcup_{\zeta, \eta\in V^{-\infty}} \operatorname{WF}_e(\pi(g)\zeta,\eta)}$$
and 
$$\operatorname{SS}(\pi)=\overline{\bigcup_{\zeta, \eta\in V^{-\infty}} \operatorname{SS}_e(\pi(g)\zeta,\eta)}$$
from Proposition \ref{wfdistvec}.

To prove the Theorem, it is therefore enough to show
$$\operatorname{WF}_e(\pi(g)\delta_H(w_2,\omega_2),\delta_H(w_1,\omega_1))=q^{-1}(\operatorname{WF}_e( \tau(h)w_1,w_2))$$
and
$$\operatorname{SS}_e(\pi(g)\delta_H(w_2,\omega_2),\delta_H(w_1,\omega_1))=q^{-1}(\operatorname{SS}_e( \tau(h)w_1,w_2)).$$
 
By Lemma \ref{identityvector}, we know $(\pi(\mu)\delta_H(w_2,\omega_2),\delta_H(w_1,\omega_1))$ is simply 
$$\langle |\det(\operatorname{Ad}(h)|_{\mathfrak{g}/\mathfrak{h}})|\cdot(\tau(h)w_1,w_2),\omega \mu\rangle.$$

Now, to compute the wave front set and singular spectrum of this generalized function, we fix a subspace $S\subset \mathfrak{g}$ such that $S\oplus \mathfrak{h}=\mathfrak{g}$. Then we can work locally in exponential coordinates $S\times \mathfrak{h}\rightarrow \mathfrak{g}$ and forget about densities (since the density bundle is locally trivial). 
In these coordinates, our generalized function is 
$$\delta_0\otimes |\det(\operatorname{Ad}(\exp Y)|_{\mathfrak{g}/\mathfrak{h}})|\cdot(\tau(\exp Y)w_1,w_2)$$ with $Y\in \mathfrak{h}$.
Now, $|\det(\operatorname{Ad}(\exp Y)|_{\mathfrak{g}/\mathfrak{h}})|$ is an analytic, nonzero function in a neighborhood of zero. Therefore, it is enough to compute the wave front set and singular spectrum of 
$$\delta_0\otimes (\tau(\exp Y)w_1,w_2).$$
Now, suppose we have open neighborhoods $0\in U_1\subset S$, $0\in U_2\subset \mathfrak{h}$ and functions $\varphi_1\in C_c^{\infty}(U_1)$, $\varphi_2\in C_c^{\infty}(U_2)$ with $\varphi_1(0)\neq 0$, $\varphi_2(0)\neq 0$. Multiplying our distribution $\delta_0\otimes (\tau(\exp Y)w_1,w_2)$ by the tensor product $\varphi_1\otimes \varphi_2$ and taking the Fourier transform yields
$$\varphi_1(0)\otimes \mathcal{F}[\varphi_2(\tau(\exp Y)w_1,w_2)].$$
The first term is never rapidly decreasing in any direction in $iS^*$ regardless of the choice of $U_1$ and $\varphi_1$. The second term is rapidly decreasing in a direction $\xi\in i\mathfrak{h}^*$ for all $\varphi_2\in C_c^{\infty}(U_2)$ for some neighborhood $0\in U_2\subset i\mathfrak{h}^*$ if and only if $\xi\notin \operatorname{WF}_e(\tau(h)w_1,w_2)$. It follows from the discussion on page 254 of \cite{Hor} that we can compute the wave front set of $\delta_0\otimes (\tau(\exp Y)w_1,w_2)$ utilizing neighborhoods of the form $U_1\times U_2$ and smooth functions of the form $\varphi_1\otimes \varphi_2$. Hence, we deduce
$$\operatorname{WF}_0(\delta_0 \otimes (\tau(\exp Y)w_1,w_2))=iS^*\times \operatorname{WF}_e(\tau(h)w_2,w_2).$$
However, this product description of the wave front set requires a non-canonical splitting of the exact sequence 
$$0\rightarrow \mathfrak{h}\rightarrow \mathfrak{g}\rightarrow \mathfrak{g}/\mathfrak{h}\rightarrow 0.$$ 
A more canonical way of writing the same thing is
$$\operatorname{WF}_e(\pi(g)\delta_H(w_2,\omega_2),\delta_H(w_1,\omega_1))=q^{-1}(\operatorname{WF}_e(\tau(h)w_1,w_2)) .$$
The first statement of Theorem \ref{inducedintro} now follows.

To compute the singular spectrum, we work in the same non-canonical, exponential coordinates. We fix precompact, open neighborhoods $0\in U_1\times U_2\subset U'_1\times U'_2\subset S\times \mathfrak{h}$ with $U_1$ (resp. $U_2$) compactly contained in $U$ (resp. $U'$). We fix $\varphi_i\in C_c^{\infty}(U_i')$ such that $\varphi_i$ is one on $U_i$ for $i=1,2$. Let 
$$\mathcal{G}_t(s)=e^{-ts^2}$$ be the standard family of Gaussians on $\mathbb{R}$. Now, we multiply 
$$\delta_0\otimes (\tau(\exp Y)w_1,w_2)$$ 
by $\varphi_1\otimes \varphi_2$ and $\mathcal{G}_t(|Z|)\otimes \mathcal{G}_t(|Y|)=\mathcal{G}_t(|Z+Y|)$ and we take the Fourier transform and evaluate at $t\zeta$ (Here we assume that $|\cdot|$ is a norm coming from an inner product for which the subspaces $S$ and $\mathfrak{h}$ are orthogonal). We obtain
$$\varphi_1(0)\otimes \mathcal{F}[\mathcal{G}_t(|Y|)\varphi_2(\tau(\exp Y)w_1,w_2)](t\zeta).$$
The first term is never exponentially decaying anywhere in $iS^*$. The second term is exponentially decaying precisely when the singular spectrum of $(\tau(\exp Y)w_1,w_2)$ does not contain $\zeta$ by definition. Thus, we obtain
$$\operatorname{SS}_0(\delta_0 \otimes (\tau(\exp Y)w_1,w_2))=iS^*\times \operatorname{SS}_e(\tau(h)w_2,w_2).$$
However, this product description of the singular spectrum requires a non-canonical decomposition $\mathfrak{g}=S\oplus \mathfrak{h}$. A more canonical way of writing the same thing is
$$\operatorname{SS}_e(\pi(g)\delta_H(w_2,\omega_2),\delta_H(w_1,\omega_1))=q^{-1}(\operatorname{SS}_e(\tau(h)w_1,w_2)).$$
The second statement of Theorem \ref{inducedintro} now follows.
\end{proof}

\section{A Parametrization of Irreducible, Tempered Representations of a Real, Reductive Algebraic Group}
\label{sec:temp}

In order to prove Theorem \ref{regularintro}, we will need a parametrization of the set of irreducible, tempered representations of a real, reductive algebraic group. In fact, a well-thought parametrization is given in section 6 of \cite{ALTV} (see also \cite{KnZ} for the original reference). For the convenience of the reader, we briefly summarize this parametrization in this section. We also prove a Lemma in this section which will be useful in subsequent sections.

Suppose $G$ is a real, reductive algebraic group with Lie algebra $\mathfrak{g}$. We say that a triple $\Gamma=(H,\gamma,R_{i\mathbb{R}}^+)$ is a \emph{Langlands parameter} for $G$ if $H\subset G$ is a Cartan subgroup with Lie algebra $\mathfrak{h}$, $\gamma$ is a level one character of the $\rho_{\text{abs}}$ double cover of $H$ (see Section 5 of \cite{ALTV} for an explanation), and $R_{i\mathbb{R}}^+$ is a choice of positive roots among the set of imaginary roots for $\mathfrak{g}$ with respect to $\mathfrak{h}$ for which $d\gamma\in \mathfrak{h}_{\mathbb{C}}^*$ is weakly dominant. In addition, the triple $\Gamma=(H,\gamma,R_{i\mathbb{R}}^+)$ must satisfy a couple of additional conditions (see Theorem 6.1 of \cite{ALTV}).

The Cartan subgroup $H=TA$ has a decomposition into a product of a maximal compact subgroup $T\subset H$ and a maximal split subgroup $A\subset H$. There is an analogous decomposition of a Langlands parameter $\Gamma=(H,\gamma,R_{i\mathbb{R}}^+)$ into discrete and continuous pieces. In particular, a \emph{discrete Langlands parameter} is a triple $\Lambda=(T,\lambda,R_{i\mathbb{R}}^+)$ for which $\lambda$ is a level one character of the $\rho_{i\mathbb{R}}$ double cover of $T$ (see section 5 of \cite{ALTV} for a definition), and $R_{i\mathbb{R}}^+$ is a choice of positive roots for the set of imaginary roots of $\mathfrak{g}$ with respect to $\mathfrak{h}$ for which $d\lambda\in \mathfrak{h}_{\mathbb{C}}^*$ is weakly dominant. This triple must satisfy a few additional assumptions (see Definition 6.5 of \cite{ALTV} for the details). If $\Lambda$ is a discrete Langlands parameter, then a \emph{continuous parameter} for $\Lambda$ is a pair $(A,\nu)$ for which $\nu$ is a (possibly nonunitary) character of $A$. The pair $(A,\nu)$ must also satisfy some additional assumptions (see Definition 6.5 of \cite{ALTV}).

If $\Gamma$ is a Langlands parameter for $G$, then $\Gamma=(H,\gamma,R_{i\mathbb{R}}^+)$ has a decomposition into a discrete Langlands parameter $\Lambda=(T,\lambda:=\gamma|_T,R_{i\mathbb{R}}^+)$ and a continuous parameter $(A,\nu:=\gamma|_A)$. If $MA$ is the Langlands decomposition of $Z_G(A)$, then there is a limit of discrete series representation $D(\Lambda)$ of $M$ associated to $\Lambda$ (see Section 9 of \cite{ALTV} or page 430 of \cite{Kn}). Then $D(\Lambda)\otimes \nu$ is a representation of $MA$. Hence, we may choose a parabolic subgroup $P=MAN$ making the real part of $\nu$ weakly dominant for the weights of $\mathfrak{a}$, the Lie algebra of $A$, in $\mathfrak{n}$, the Lie algebra of $N$, and we may extend $D(\Lambda)\otimes \nu$ trivially on $N$ to a representation of $P$. The representation of $G$ induced from the representation $D(\Lambda)\otimes \nu$ of $P$ is denoted by $I(\Lambda)$.

The (possibly nonunitary) representation $I(\Gamma)$ has a Langlands quotient $J(\Gamma)$ which is an irreducible (possibly nonunitary) representation of $G$. In addition, every irreducible (possibly nonunitary) representation of $G$ is of the form $J(\Gamma)$ for some Langlands parameter $\Gamma$, and two parameters $\Gamma$ and $\Gamma_1$ are $G$ conjugate if, and only if the two irreducible representations $J(\Gamma)$ and $J(\Gamma_1)$ are isomorphic (see \cite{La} for the original reference and Section 6 of \cite{ALTV}, Chapter 6 of \cite{V2}, and Chapters 8 and 14 of \cite{Kn} for expositions). If $\Gamma$ has discrete part $\Lambda$ and continuous part $(A,\nu)$, then we also write $J(\Lambda,\nu)$ for $J(\Gamma)$. The irreducible representation $J(\Lambda,\nu)$ is tempered if, and only if the character $\nu$ is unitary. Every irreducible, tempered representation of $G$ is unitary.

We write $\widehat{G}$ for the set of irreducible, unitary representations of $G$ equipped with the Fell topology. We denote by $\widehat{G}_{\text{temp}}$ the closed subspace of irreducible, tempered representations. If $H\subset G$ is a Cartan subgroup of $G$ with Lie algebra $\mathfrak{h}$ and $R_{i\mathbb{R}}^+$ is a choice of positive, imaginary roots for $\mathfrak{g}$ with respect to $\mathfrak{h}$, then we have a corresponding closed Weyl chamber
$$i\mathfrak{h}_+^*=\left\{\lambda\in i\mathfrak{h}^*|\ \langle \lambda,\alpha^{\vee}\rangle\geq 0\ \text{if}\ \alpha\in R_{i\mathbb{R}}^+\right\}$$
in $i\mathfrak{h}^*$ where whenever $\alpha\in R_{i\mathbb{R}}^+$ is an imaginary root of $\mathfrak{g}$ with respect to $\mathfrak{h}$, we denote by $\alpha^{\vee}$ the corresponding coroot. Then we may define
$$\widehat{G}_{\text{temp},i\mathfrak{h}^*_+}$$
to be the set of irreducible, tempered representations of the form $J(\Gamma)$ where $\Gamma=(H,\gamma,R_{i\mathbb{R}}^+)$ and $(H,R_{i\mathbb{R}}^+)$ corresponds to the closed Weyl chamber $i\mathfrak{h}^*_+$. We therefore obtain 
$$\widehat{G}_{\text{temp}}=\bigcup_{i\mathfrak{h}^*_+}\widehat{G}_{\text{temp},i\mathfrak{h}^*_+}$$
where the union is over $G$ conjugacy classes of closed Weyl chambers $i\mathfrak{h}^*_+$ in the dual of $i$ times the collection of purely imaginary valued linear functionals on a Cartan subalgebra $\mathfrak{h}\subset \mathfrak{g}$. 

Now, if $\pi$ is a unitary representation of $G$ that is weakly contained in the regular representation, then we may write $\pi$ as a direct integral of irreducibe, tempered representations
$$\pi \cong \int^{\oplus} \sigma^{\oplus m(\pi,\sigma)}d\mu_{\pi}$$
with respect to a positive measure $\mu$ on $\widehat{G}_{\text{temp}}$ (see for instance Chapter 8 of \cite{Di} for a proof that this is always possible) and a function $m(\pi,\sigma)$ that keeps track of the multiplicity of $\sigma$ in $\pi$. For each closed Weyl chamber $i\mathfrak{h}^*_+\subset i\mathfrak{h}^*$ in $i$ times the dual of a Cartan subalgebra $\mathfrak{h}\subset \mathfrak{g}$, define $$\pi_{i\mathfrak{h}^*_+}\cong \int_{\sigma \in \widehat{G}_{\text{temp},i\mathfrak{h}^*_+}}\sigma^{\oplus m(\pi,\sigma)} d\mu_{\pi}|_{\widehat{G}_{\text{temp},i\mathfrak{h}^*_+}}.$$
Then we have an isomorphism
$$\pi\cong \bigoplus_{i\mathfrak{h}^*_+} \pi_{i\mathfrak{h}^*_+}$$
where the sum is over the finite collection of $G$ conjugacy classes of closed Weyl chambers in $i\mathfrak{h}^*$ for some Cartan subalgebra $\mathfrak{h}\subset \mathfrak{g}$.

\begin{lemma} \label{Lem:WeylChamber_Reduction} If $\pi$ is a unitary representation of a real, reductive algebraic group $G$ that is weakly contained in the regular representation, then 
$$\operatorname{WF}(\pi)=\bigcup_{i\mathfrak{h}^*_+} \operatorname{WF}(\pi_{i\mathfrak{h}^*_+}),\ \operatorname{SS}(\pi)=\bigcup_{i\mathfrak{h}^*_+} \operatorname{SS}(\pi_{i\mathfrak{h}^*_+})$$
where the unions are over the finite collection of $G$ conjugacy classes of closed Weyl chambers in $i\mathfrak{h}^*$ for some Cartan subalgebra $\mathfrak{h}\subset \mathfrak{g}$. In addition, we have
$$\operatorname{AC}\left(\mathcal{O}\operatorname{-}\operatorname{supp}\pi\right)=\bigcup_{i\mathfrak{h}^*_+} \operatorname{AC}\left(\mathcal{O}\operatorname{-}\operatorname{supp}\pi_{i\mathfrak{h}^*_+}\right)$$
where the union is again over the finite collection of $G$ conjugacy classes of closed Weyl chambers in $i\mathfrak{h}^*$ for some Cartan subalgebra $\mathfrak{h}\subset \mathfrak{g}$.
\end{lemma}

The first statement of the Lemma regarding wave front sets follows from part (b) of Proposition 1.3 on page 121 of \cite{How}. The argument in the singular spectrum case is analogous. To show the second statement, it is enough to show that whenever $S_1,\ldots,S_n\subset W$ is a finite collection of subsets of a finite-dimensional, real vector space $W$, we have 
$$\operatorname{AC}\left(\bigcup_{i=1}^nS_i\right)=\bigcup_{i=1}^n\operatorname{AC}(S_i).$$
Indeed, $S_i\subset \cup S_i$ implies $\operatorname{AC}(S_i)\subset \operatorname{AC}(\cup S_i)$ and the right hand side is contained in the left hand side. To show the opposite inclusion, suppose $\xi$ is in the set on the left. Fix a norm on $W$, and define
$$\Gamma_{\epsilon}=\{\eta\in W|\ |t\eta-\xi|<\epsilon\ \text{some}\ t>0\}$$
for every $\epsilon>0$. Since $\xi$ is in the set on the left,
$\Gamma_{\epsilon}\cap \bigcup_{i=1}^nS_i$
is unbounded. But, then certainly $\Gamma_{\epsilon} \cap S_i$ is unbounded for some $i$. Let the subcollection $I_{\epsilon}\subset \{1,\ldots,n\}$ be the set of $i$ such that $\Gamma_{\epsilon} \cap S_i$ is unbounded. Now, $I_{\epsilon}$ is non-empty for every $\epsilon>0$ and $I_{\epsilon'}\subset I_{\epsilon}$ if $\epsilon'<\epsilon$. One deduces that there is some $i$ in every $I_{\epsilon}$ and $\xi\in \operatorname{AC}(S_i)$ for this particular $i$. The statement follows.

Since $$\mathcal{O}\operatorname{-}\operatorname{supp}\pi=\bigcup_{i\mathfrak{h}^*_+} \left(\mathcal{O}\operatorname{-}\operatorname{supp}\pi_{i\mathfrak{h}^*_+}\right),$$
we deduce $$\operatorname{AC}\left(\mathcal{O}\operatorname{-}\operatorname{supp}\pi\right)=\operatorname{AC}\left(\bigcup_{i\mathfrak{h}^*_+} \left(\mathcal{O}\operatorname{-}\operatorname{supp}\pi_{i\mathfrak{h}^*_+}\right)\right).$$
The Lemma follows.

\section{Wave Front Sets of Pieces of the Regular Representation Part I}
\label{sec:regular_part1}

Our next task is to prove Theorem \ref{regularintro}. Suppose $G$ is a real, reductive algebraic group, and suppose $\pi$ is weakly contained in the regular representation of $G$. Then we must show
$$\operatorname{SS}(\pi)=\operatorname{WF}(\pi)=\operatorname{AC}\left(\mathcal{O}\operatorname{-}\operatorname{supp}\pi\right).$$
However, given that $\operatorname{SS}(\pi)\supset \operatorname{WF}(\pi)$, it is enough to show
$$\operatorname{WF}(\pi)\supset \operatorname{AC}\left(\mathcal{O}\operatorname{-}\operatorname{supp}\pi\right)$$
and 
$$\operatorname{SS}(\pi)\subset \operatorname{AC}\left(\mathcal{O}\operatorname{-}\operatorname{supp}\pi\right).$$

This section will be devoted to proving the first inclusion. The next section will be devoted to proving the second inclusion.

\begin{proposition} \label{regularwf} Suppose $G$ is a real, reductive algebraic group, and suppose $\pi$ is a unitary representation of $G$ that is weakly contained in the regular representation of $G$. Then
$$\operatorname{WF}(\pi)\supset \operatorname{AC}\left(\mathcal{O}\operatorname{-}\operatorname{supp}\pi\right).$$
\end{proposition}

By Lemma \ref{Lem:WeylChamber_Reduction}, in order to prove Proposition \ref{regularwf}, it is enough to show
$$\operatorname{WF}(\pi_{i\mathfrak{h}^*_+})\supset \operatorname{AC}(\mathcal{O}\operatorname{-}\operatorname{supp}\pi_{i\mathfrak{h}^*_+})$$
for every closed Weyl chamber in $i$ times the dual of a Cartan subalgebra. In particular, we may assume that there exists a Cartan subgroup $H\subset G$ and a choice of positive imaginary roots $R_{i\mathbb{R}}^+$ for $\mathfrak{g}$ with respect to $\mathfrak{h}$ such that all of the irreducible representations in the support of $\pi$ are of the form $J(\Gamma)$ with Langlands parameter $\Gamma=(H,\gamma,R_{i\mathbb{R}}^+)$ for some $\gamma$.

Next, we note that in the direct integral decomposition of $\pi$, the positive measure $\mu_{\pi}$ on $\widehat{G}_{\text{temp},i\mathfrak{h}^*_+}$ is only well-defined up to an equivalence relation. Here two measures are equivalent if, and only if they are absolutely continuous with respect to each other. After multiplying $\mu_{\pi}$ by a suitable non-vanishing function on $\widehat{G}_{\text{temp}}$, one can find a measure $\mu_{\pi}'$ which is equivalent to $\mu_{\pi}$ and has finite volume. Therefore, we will assume without loss of generality in what follows that $\mu_{\pi}$ has finite volume.

Next, we introduce a continuous map with finite fibers 
$$\widehat{G}_{\text{temp},i\mathfrak{h}^*_+}\longrightarrow i\mathfrak{h}^*_+$$
via 
$$J(\Gamma)\mapsto d\gamma$$
for every Langlands parameter $\Gamma=(H,\gamma,R_{i\mathbb{R}}^+)$ with $d\gamma\in i\mathfrak{h}^*_+$. In particular, we may take our finite, positive measure $\mu_{\pi}$ on $\widehat{G}_{\text{temp},i\mathfrak{h}^*_+}$ and push it forward to a finite, positive measure on $i\mathfrak{h}_+^*$. From now on, we will abuse notation and write $\mu_{\pi}$ for the measure on both spaces.

Let $j_G$ be the Jacobian of the exponential map $\exp\colon \mathfrak{g}\rightarrow G$ in a neighborhood of the identity; we normalize the Lebesgue measure on $\mathfrak{g}$ and the Haar measure on $G$ so that $j_G(0)=1$. Then $j_G$ extends to an analytic function on $\mathfrak{g}$. Moreover, it has a unique analytic square root $j_G^{1/2}$ with $j_G^{1/2}(0)=1$. 

\begin{lemma} \label{tempereddist} Let $\mu_{\pi}$ be a finite, positive measure on $i\mathfrak{h}_+^*$. For each $\sigma\in \widehat{G}_{\textup{temp}}$, let $\Theta_{\sigma}$ denote the Harish-Chandra character of $\sigma$ and let $$\theta_{\sigma}=(\exp^*\Theta_{\sigma})j_G^{1/2}$$
denote the Lie algebra analogue of the character of $\sigma$. If $f\in L^1(\widehat{G}_{\textup{temp},i\mathfrak{h}^*_+},\mu_{\pi})$ is a positive $L^1$ function, then 
$$\int_{\sigma\in \widehat{G}_{\textup{temp},i\mathfrak{h}^*_+}}\theta_{\sigma}f\mu_{\pi}$$
defines a tempered distribution on $\mathfrak{g}$.
\end{lemma}

In order to prove the Lemma, we require some remarks on the canonical measure on a coadjoint orbit. Let $G$ be a Lie group with Lie algebra $\mathfrak{g}$, and let $G\cdot \xi=\mathcal{O}_{\xi}\subset i\mathfrak{g}^*$ be the coadjoint orbit through $\xi\in i\mathfrak{g}^*$. Define a 2-form on $\mathcal{O}_{\xi}$ by
$$\omega_{\eta}(\operatorname{ad}_{\eta}^*X,\operatorname{ad}_{\eta}^*Y)=-\eta([X,Y])$$
for every $X,Y\in \mathfrak{g}$ and $\eta\in \mathcal{O}_{\xi}$. This 2-form makes $\mathcal{O}_{\xi}$ into a symplectic manifold (see for instance page 139 of \cite{CdS}), and the absolute value of the top dimensional form
$$\eta\mapsto \frac{\omega_{\eta}^{\wedge \dim\mathcal{O}_{\eta}/2}}{(\dim\mathcal{O}_{\xi}/2)!(2\pi)^{\dim\mathcal{O}_{\eta}/2}}$$
defines an invariant smooth density on $\mathcal{O}_{\xi}$, which is called the \emph{canonical measure} on $\mathcal{O}_{\xi}$. We will sometimes denote this density by $\eta\mapsto m(\mathcal{O}_{\xi})_{\eta}$ and we will sometimes abuse notation and simply denote by $\mathcal{O}_{\xi}$ both the orbit and the canonical invariant measure on the orbit.

\begin{proof} We will show that the above integral defines a tempered distribution on $\mathfrak{g}$ by showing that it is the Fourier transform of a tempered distribution on $i\mathfrak{g}^*$. For each $\sigma\in \widehat{G}_{\text{temp},i\mathfrak{h}^*_+}$, let $\mathcal{O}_{\sigma}$ denote the canonical invariant measure on the finite union of coadjoint orbits associated to $\sigma$ \cite{R1}, \cite{R2}. We will show that the integral
$$\int_{\sigma\in \widehat{G}_{\text{temp},i\mathfrak{h}^*_+}}\mathcal{O}_{\sigma}f \mu_{\pi}$$
defines a tempered distribution on $i\mathfrak{g}^*$ and its Fourier transform is the integral in the statement of the Lemma. Following Harish-Chandra we define a map
$$\psi\colon C_c^{\infty}(i\mathfrak{g}^*)\rightarrow C_c^{\infty}((i\mathfrak{h}_+^*)')$$
via $$\psi\colon \varphi\mapsto (\lambda\mapsto \langle \mathcal{O}_{\lambda},\varphi\rangle).$$
Here $\mathcal{O}_{\lambda}$ denotes the canonical invariant measure on the orbit $G\cdot \lambda$, which by an abuse of notation we will also denote by $\mathcal{O}_{\lambda}$. Further $(i\mathfrak{h}_+^*)'\subset i\mathfrak{h}^*$ is the set of regular elements in $i\mathfrak{h}_+^*$.
Harish-Chandra showed that this map, which he called the invariant integral, extends to a continuous map on spaces of Schwartz functions \cite{HC}. Moreover, he showed that functions in the image extend uniquely to all of $i\mathfrak{h}_+^*$ (see page 576 of \cite{HC2}). Thus, we obtain a continuous map
$$\psi\colon \mathcal{S}(i\mathfrak{g}^*)\rightarrow \mathcal{S}(i\mathfrak{h}_+^*).$$
Now, if the infinitesimal character of $\sigma$ is regular, then $\mathcal{O}_{\sigma}=\mathcal{O}_{\lambda}$ with $\lambda\in (i\mathfrak{h}_+^*)'$. Therefore,
$$\langle \mathcal{O}_{\sigma},\varphi\rangle= \delta_{\lambda}\circ \psi.$$
If the infinitesimal character of $\sigma$ is singular, then $\mathcal{O}_{\sigma}$ can be written as a limit 
$$\mathcal{O}_{\sigma}=\lim_{\lambda\in (i\mathfrak{h}^*_+)', \lambda\rightarrow \lambda_0}\mathcal{O}_{\lambda}$$
where $\lambda_0\in i\mathfrak{h}_+^*$ is singular \cite{R2}, \cite{R3}. Therefore,
$$\mathcal{O}_{\sigma}=\delta_{\lambda_0}\circ \psi$$
for some $\lambda_0\in i\mathfrak{h}_+^*$. Now, the map 
$$\varphi\mapsto \int_{\sigma\in \widehat{G}_{\text{temp},\mathfrak{h}^*_+}} f(\sigma) \langle \mathcal{O}_{\sigma},\varphi\rangle \mu_{\pi}$$
for $\varphi\in C_c^{\infty}(i\mathfrak{g}^*)$ is simply the map
$$\varphi\mapsto \int_{i\mathfrak{h}^*_+}\psi(\varphi)f \mu_{\pi}.$$
Here we have, by an abuse of notation, written $f$ for the pushforward of $f$ under the map
$$\widehat{G}_{\text{temp},i\mathfrak{h}^*}\rightarrow i\mathfrak{h}^*.$$
Note that $f\mu_{\pi}$ defines a tempered distribution on $i\mathfrak{h}^*_+$ since it is a positive, finite measure on $i\mathfrak{h}^*_+$. 
Since $\psi$ is a continuous map between Schwartz spaces and $f\mu_{\pi}$ is a tempered distribution on $i\mathfrak{h}_+^*$, we conclude that $$\int_{\widehat{G}_{\text{temp},i\mathfrak{h}^*_+}} \mathcal{O}_{\sigma}f\mu_{\pi}$$ is a tempered distribution on $i\mathfrak{g}^*$. Now, the Fourier transform of this tempered distribution is defined by
$$\omega\mapsto \langle \int_{\widehat{G}_{\text{temp},i\mathfrak{h}^*_+}} \mathcal{O}_{\sigma}f\mu_{\pi},\mathcal{F}[\omega]\rangle=\int_{\widehat{G}_{\text{temp},i\mathfrak{h}^*_+}} \langle \theta_{\sigma},\omega\rangle f \mu_{\pi}$$
for any smooth, compactly supported density $\omega$ on $i\mathfrak{g}^*$. Here we have used $\mathcal{F}[\mathcal{O}_{\sigma}]=\theta_{\sigma}$, which was proved by Rossmann \cite{R1}, \cite{R2}. Thus, the integral is the Fourier transform of a tempered distribution and is therefore tempered.
\end{proof}

Suppose $\Lambda=(T,\lambda,R_{i\mathbb{R}}^+)$ is a discrete Langlands parameter for $T\subset H$, the maximal compact subgroup of a Cartan subgroup $H$, suppose $\mathfrak{h}$ is the Lie algebra of $H$, and suppose $i\mathfrak{h}_+^*$ is the closed Weyl chamber defined by the choice of positive imaginary roots $R_{i\mathbb{R}}^+$. Then $$\widehat{G}_{\text{temp},i\mathfrak{h}^*_+,\Lambda}$$
is the collection of irreducible, tempered representations of the form $J(\Lambda,\nu)$ for some unitary character $\nu$ of $A$. Further, we have an embedding
$$\widehat{G}_{\text{temp},i\mathfrak{h}^*_+,\Lambda}\hookrightarrow i\mathfrak{a}^*$$
by 
$$J(\Lambda,\nu)\mapsto d\nu.$$
We will denote the image of the above map by $\mathfrak{a}^*_{\Lambda}$.

\begin{lemma} \label{wfcharacter} Suppose $\pi$ is a unitary representation of $G$ that is weakly contained in the regular representation, and suppose that $\pi$ decomposes as a direct integral of irreducible representations with respect to the finite, positive measure $\mu_{\pi}$ on $\widehat{G}_{\textup{temp},i\mathfrak{h}^*_+}$. Let $H$ be the Cartan subgroup with Lie algebra $\mathfrak{h}$, and let $R_{i\mathbb{R}}^+$ be the choice of positive imaginary roots that determines the Weyl chamber $i\mathfrak{h}^*_+$. Let $H=TA$ be the decomposition of $H$ into compact and split pieces, and let $\mathfrak{t}$ denote the Lie algebra of $T$. Assume that there exists a polynomial $p$ on $i\mathfrak{t}^*$ such that for every discrete Langlands parameter $\Lambda=(T,\lambda,R_{i\mathbb{R}}^+)$, we have
$$\int_{\nu\in i\mathfrak{a}^*} d\mu_{\pi}|_{\widehat{G}_{\textup{temp},i\mathfrak{h}^*_+,\Lambda}}\leq |p(d\lambda)|.$$
Then 
$$\operatorname{WF}(\pi)\supset \operatorname{WF}_e\left(\int_{\sigma\in \widehat{G}_{\textup{temp},i\mathfrak{h}^*_+}}\Theta_{\sigma}\mu_{\pi}\right).$$
From this, we immediately deduce
$$\operatorname{WF}(\pi)\supset \operatorname{WF}_0\left(\int_{\sigma\in \widehat{G}_{\textup{temp},i\mathfrak{h}^*_+}}\theta_{\sigma}\mu_{\pi}\right).$$
\end{lemma}

\begin{proof} First, we note that our hypothesis and Lemma \ref{tempereddist} together with the relation $\exp^*\Theta_{\sigma}=\theta_{\sigma}j_G^{-1/2}$ imply that the above integral defines a distribution in a neighborhood of the identity $e\in G$. Therefore, the right hand side is at least well defined. 

Now, let us break up the integral
$$\int_{\sigma\in \widehat{G}_{\text{temp},i\mathfrak{h}^*_+}}\Theta_{\sigma}\mu_{\pi}=\sum_{\Lambda} \int_{d\nu\in i\mathfrak{a}_{\Lambda}^*}\Theta_{J(\Lambda,\nu)}\mu_{\pi}|_{\widehat{G}_{\text{temp},i\mathfrak{h}^*_+,\Lambda}}$$
where the sum is over discrete Langlands parameters $\Lambda=(T,\lambda,R_{i\mathbb{R}}^+)$ for which $T\subset H$ is a maximal compact subgroup and the choice of positive roots $R_{i\mathbb{R}}^+$ determines the Weyl chamber $\mathfrak{h}^*_+$.
If $V(\Lambda,\nu)$ denotes the Hilbert space on which $J(\Lambda,\nu)$ acts, then utilizing the compact picture for induced representations (see page 169 of \cite{Kn}), for fixed discrete Langlands parameter $\Lambda$, we may identify the spaces $V(\Lambda,\nu)$ varying over all unitary characters $\nu$ of $A$ as unitary representations of $K$. 
Thus, for a fixed discrete Langlands parameter $\Lambda$, we may fix an orthonormal basis for $V(\Lambda,\nu)$ that is independent of $d\nu\in i\mathfrak{a}_{\Lambda}^*$, which we will call $\{e_{\tau,i}^{\Lambda}(\nu)\}$. We choose this basis in such a way that each vector $e^{\Lambda}_{\tau,i}(\nu)$ is contained in the isotypic component of $\tau\in \widehat{K}$. Now, since 
$$\pi\simeq \sum_{\Lambda} \int_{d\nu\in i\mathfrak{a}_{\Lambda}^*}J(\Lambda,\nu)^{\oplus m(\pi,J(\Lambda,\nu))}d\mu_{\pi}|_{\widehat{G}_{\text{temp},i\mathfrak{h}^*_+,\Lambda}},$$
the representation 
$$\sum_{\Lambda}\int_{d\nu\in i\mathfrak{a}^*_{\Lambda}}J(\Lambda,\nu) d\mu_{\pi}|_{\widehat{G}_{\text{temp},i\mathfrak{h}^*_+,\Lambda}}$$
is a subrepresentation of our representation $\pi$. Now, the map $\nu\mapsto e_{\tau,i}^{\Lambda}(\nu)$ is contained in the above direct integral representation since the measure $\mu_{\pi}|_{\widehat{G}_{\text{temp},i\mathfrak{h}^*_+,\delta}}$ is finite. Thus, for fixed $i$, we may view $e_{\tau,i}^{\Lambda}(\nu)$ as a vector in our representation $\pi$. Now, we observe that the weighted sum of matrix coefficients
$$ \sum_{\Lambda}\sum_{i,\tau}\int_{\nu\in i\mathfrak{a}_{\Lambda}^*}(J(\Lambda,\nu)(g)e_{\tau,i}^{\Lambda}(\nu),e_{\tau,i}^{\Lambda}(\nu))\mu_{\pi}|_{\widehat{G}_{\text{temp},i\mathfrak{h}^*_+,\Lambda}}$$
is simply our integral $$\int_{\sigma\in \widehat{G}_{\text{temp},i\mathfrak{h}^*_+}}\Theta_{\sigma}\mu_{\pi}.$$

Let $V$ denote the Hilbert space on which $\pi$ acts, and let $P$ be the orthogonal projection of $V$ onto the subspace generated by the vectors $\{e_{\tau,i}^{\Lambda}\}$. Define $T_N=(I+\Omega_K)^{-N}P$ where $\Omega_K$ is the Casimir operator for $K$. 

First, observe $$\operatorname{Tr}(\pi(g)P)=\int_{\sigma\in \widehat{G}_{\text{temp},i\mathfrak{h}^*_+}}\Theta_{\sigma}\mu_{\pi}$$
as a distribution. Next, we claim that $T_N$ is a trace class operator for sufficiently large $N$.

Observe
$$\left|(I+\Omega_K)^{-N}P\right|_1=\sum_{\Lambda}\sum_{i,\tau}\frac{1}{(1+|\tau|^2)^N}\left|\int_{d\nu\in i\mathfrak{a}_{\Lambda}^*}d\mu_{\pi}|_{\widehat{G}_{\text{temp},i\mathfrak{h}^*_+,\Lambda}}\right|$$
$$\leq \sum_{\Lambda=(T,\lambda,R_{i\mathbb{R}}^+)}\sum_{i,\tau}\frac{1}{(1+|\tau|^2)^N}|p(d\lambda)|$$

\noindent where $|\cdot|_1$ denotes the norm on the Banach space of trace class operators. We recall that the multiplicity of $\tau$ in any irreducible $J(\Lambda,\nu)$ is at most $(\dim\tau)^2$ (see page 205 of \cite{Kn} for an exposition or \cite{HC3} for the original reference). Now, fix an inner product on the vector space $i\mathfrak{t}^*$, and let $|\cdot|$ be the associated norm. By Weyl's dimension formula, we have $(\dim\tau)^2\leq C(1+|\tau|^2)^r$ where $r$ is the number of positive roots of $K$ with respect to a maximal torus and $C$ is a positive constant. Moreover, a limit of discrete series $J(\Lambda,\nu)$ can only contain $\tau$ as a $K$ type if $|d\lambda|\leq |\tau|+C_1$ where $C_1>0$ is a constant independent of $\tau$ (see page 460 of \cite{Kn} for an exposition and \cite{HS} for the original reference). Counting lattice points, this means that the number of such $\delta$ is bounded by $C_2(1+|\tau|^2)^k$ where $k$ is the rank of $G$ and $C_2>0$ is a positive constant. The relationship between $|d\lambda|$ and $|\tau|$ also implies that we may bound $|p(d\lambda)|\leq C_3(1+|\tau|^2)^M$ for some positive integer $M$ and some constant $C_3>0$ whenever $\tau$ is a $K$ type of $J(\Lambda,\nu)$. Combining these facts, the above expression becomes
$$\leq CC_2C_3\sum_{\tau}\frac{(1+|\tau|^2)^{r+k+M}}{(1+|\tau|^2)^N}.$$ If $N$ is sufficiently large, this sum will converge and therefore $(I+\Omega_K)^{-N}P$ is a trace class operator on $V$. Now, using Howe's original definition of the wave front set involving trace class operators (see Proposition \ref{howeusequiv}), we observe
$$WF(\pi)\supset \operatorname{WF}_e\left(\operatorname{Tr}(\pi(g)(I+\Omega_K)^{-N}P)\right).$$

To finish the argument, we first recall 
$$\langle \int_{\sigma\in \widehat{G}_{\text{temp},i\mathfrak{h}^*_+}}\Theta_{\sigma}\mu_{\pi},\omega\rangle=\operatorname{Tr}(\pi(\omega)P)$$
for any smooth, compactly supported density $\omega$ on $\mathfrak{g}$. Then we observe
$$\operatorname{Tr}(\pi(\omega)P)=\operatorname{Tr}(\pi(\omega)(I+\Omega_K)^{N}(I+\Omega_K)^{-N}P)$$
$$=\operatorname{Tr}(\pi(L_{(I+\Omega_K)^N}\omega)(I+\Omega_K)^{-N}P)=L_{(I+\Omega_K)^N}\operatorname{Tr}(\pi(\omega)(I+\Omega_K)^{-N}P).$$

Since differential operators can only decrease the wave front set, we obtain
$$WF(\pi)\supset \operatorname{WF}_e\left(\int_{\sigma\in \widehat{G}_{\text{temp},i\mathfrak{h}^*_+}}\Theta_{\sigma}\mu_{\pi}\right)$$
and the Lemma has been verified.
\end{proof} 

Next, we need a Lemma involving the canonical measure on regular, coadjoint orbits. We fix an arbitrary inner product $(\cdot,\cdot)$ on $i\mathfrak{g}^*$, and we denote by $|\cdot|$ the corresponding norm. If $M\subset i\mathfrak{g}^*$ is any submanifold we denote by $\operatorname{Eucl}(M)$ the following density on $M$. If $\xi\in M$ and $\dim T_{\xi}M=k$, we fix an orthonormal basis $e_1,\ldots,e_k$ of $T_{\xi}M$, and for every $v_1,\ldots,v_k\in T_{\xi}M$, we define
$$\operatorname{Eucl}(M)_{\xi}(v_1,\ldots,v_k)=\left|\det((v_i,e_j)_{i,j})\right|.$$
One notes that this definition is independent of the orthonormal basis $\{e_j\}$.

\begin{lemma} \label{canonicalmeasure} Let $G$ be a Lie group, and let $i\mathfrak{g}^*$ be $i$ times the dual of the Lie algebra of $G$. If $\xi\in i\mathfrak{g}^*$, let $m(\mathcal{O}_{\xi})$ denote the canonical measure on the $G$ orbit through $\xi$ and let $\operatorname{Eucl}(\mathcal{O}_{\xi})$ denote the measure on the $G$ orbit through $\xi$ that is induced from a fixed inner product on $i\mathfrak{g}^*$. For every $\xi\in \mathcal{O}_{\xi}$, we have
$$F(\xi)m(\mathcal{O}_{\xi})_{\xi}=\operatorname{Eucl}(\mathcal{O}_{\xi})_{\xi}$$
for some function $F$ on $i\mathfrak{g}^*$. Then there exists a positive constant $C>0$ (depending on $G$) such that
$$|F(\xi)|\leq C(1+|\xi|)^{\dim G/2}$$
for all $\xi\in i\mathfrak{g}^*$. 
\end{lemma}
 
\begin{proof} In order to simplify our notation, we prove the Lemma for coadjoint orbits in $\mathcal{O}_{\xi}$ contained in $\mathfrak{g}^*$ instead of $i\mathfrak{g}^*$. Multiplying by $i$ everywhere, one will obtain the above Lemma. Observe that we must define the 2 form $\omega_{\xi}$ on the coadjoint orbit $G\cdot \xi=\mathcal{O}_{\xi}\subset \mathfrak{g}^*$ (instead of $i\mathfrak{g}^*$) by 
$$\omega_{\xi}(\operatorname{ad}_{\xi}^*X,\operatorname{ad}_{\xi}^*Y)=\xi([X,Y])$$
(dividing by $i$ twice removes the negative sign).

Now, fix $\xi\in \mathfrak{g}^*$, and choose a basis $\{\eta_1,\ldots,\eta_k\}$ of $T_{\xi}\mathcal{O}_{\xi}$ that is orthonormal with respect to the restriction of the inner product on $\mathfrak{g}^*$ to $T_{\xi}\mathcal{O}_{\xi}$. For $i=1,\ldots,k$, define $X_i\in \mathfrak{g}$ by $\eta(X_i)=(\eta,\eta_i)$ for all $\eta\in \mathfrak{g}^*$. Note that we also have $(X_i,W)=\eta_i(W)$ for all $W\in \mathfrak{g}$ (where the inner product on $\mathfrak{g}$ is the one induced from our fixed inner product on $\mathfrak{g}^*$). We claim that $\operatorname{ad}_{X_1}^*\xi,\ldots,\operatorname{ad}_{X_{k}}^*\xi$ is a basis of $T_{\xi}\mathcal{O}_{\xi}$. To show this, we need only show that $\{X_i\}$ is a linearly independent set in $\mathfrak{g}/Z_{\mathfrak{g}}(\xi)$. Write $\eta_i=\operatorname{ad}_{Y_i}^*\xi$. If $W\in Z_{\mathfrak{g}}(\xi)$, then $$(X_i,W)=\eta_i(W)=\operatorname{ad}_{Y_i}^*\xi(W)=-\operatorname{ad}_W^*\xi(Y_i)=0.$$
Since each $X_i$ is orthogonal to $Z_{\mathfrak{g}}(\xi)$, the set $\{X_i\}$ must remain linearly independent in $\mathfrak{g}/Z_{\mathfrak{g}}(\xi)$.

Next, we compute
$$\operatorname{Eucl}(\mathcal{O}_{\xi})_{\xi}(\operatorname{ad}_{X_1}^*\xi,\ldots,\operatorname{ad}_{X_k}^*\xi)=\left|\det((\operatorname{ad}_{X_i}^*\xi,\eta_j))\right|$$
and $$m(\mathcal{O}_{\xi})_{\xi}(\operatorname{ad}_{X_1}^*\xi,\ldots,\operatorname{ad}_{X_k}^*\xi)=c\left|\det(\xi([X_i,X_j]))\right|^{1/2}$$
$$=c\left|\det(\operatorname{ad}_{X_i}^*\xi(X_j))\right|^{1/2}=c\left|\det((\operatorname{ad}_{X_i}^*\xi,\eta_j))\right|^{1/2}$$
where $$c=\frac{1}{(2\pi)^{\dim\mathcal{O}_{\xi}/2}}.$$
Thus, we obtain
$$F(\xi)=\frac{1}{c}\left|\det((\operatorname{ad}_{X_i}^*\xi,\eta_j))\right|^{1/2}.$$ 
Now, we note that $$\mathfrak{g}\otimes \mathfrak{g}^*\rightarrow \mathfrak{g}^*,\ \text{by}\ (X,\xi)\mapsto \operatorname{ad}^*_X\xi$$
is a linear map between finite-dimensional vector spaces. In particular, it is a bounded, linear map, and there exists a constant $C_1$ (depending on $G$) such that 
$$|\operatorname{ad}_X^*\xi|\leq C_1|X||\xi|\ \text{for\ all}\ X\in \mathfrak{g},\ \xi\in \mathfrak{g}^*.$$
Therefore, we estimate,
$$\left|\det((\operatorname{ad}_{X_i}^*\xi,\eta_j))\right|\leq (\dim G)^2 \prod_{i=1}^k C_1 |X_i||\xi|=(\dim G)^2C_1^k|\xi|^{k/2}.$$
And for $c_k=(1/c)(\dim G)C_1^{k/2}$, we obtain
$$|F(\xi)|\leq c_k|\xi|^k$$
whenever $\dim\mathcal{O}_{\xi}=k$. Since the dimension of every coadjoint orbit is less than or equal to the dimension of $G$, we obtain
$$|F(\xi)|\leq C(1+|\xi|)^{\dim G/2}$$
where $C$ is the maximum of the constants $c_k$. The Lemma follows.
\end{proof}

Next, we prove Proposition \ref{regularwf}.

\begin{proof} Suppose $\xi\in \operatorname{AC}(\mathcal{O}\operatorname{-}\operatorname{supp}\pi)$. We must show $\xi\in \operatorname{WF}(\pi)$. As in the last Lemma, we fix an inner product $(\cdot,\cdot)$ on $i\mathfrak{g}^*$, and we let $|\cdot|$ denote the corresponding norm. Without loss of generality, we may assume $|\xi|=1$. By Lemma \ref{wfcharacter}, to show $\xi\in \operatorname{WF}(\pi)$, it is enough to show 
$$\xi\in \operatorname{WF}_0\left(\int_{\sigma\in \widehat{G}_{\text{temp},i\mathfrak{h}^*_+}}\theta_{\sigma}d\mu_{\pi}'\right)$$
for some finite positive measure $\mu_{\pi}'$ that is equivalent to $\mu_{\pi}$. Now, to check this fact, we fix an even Schwartz function $\mathcal{F}[\varphi]\in \mathcal{S}(i\mathfrak{g}^*)$ such that $\mathcal{F}[\varphi](x)\geq 0$ for all $x$ and $\mathcal{F}[\varphi](x)=1$ if $|x|\leq 1$. Then $\mathcal{F}[\varphi]$ is the Fourier transform of an even Schwartz function $\varphi\in \mathcal{S}(\mathfrak{g})$.

By Theorem 3.22 on page 155 of \cite{Fo}, if $\xi$ is not in the wave front set of $$\int_{\sigma\in \widehat{G}_{\text{temp},i\mathfrak{h}^*_+}}\theta_{\sigma}d\mu_{\pi}'$$ at 0, then there must exist an open cone $\xi\in \mathcal{C}$ such that for $\eta\in \mathcal{C}$ with $||\xi|-|\eta||<\epsilon$, there exist constants $C_{N,\epsilon}$ for every $0<\epsilon<1$ and $N\in \mathbb{N}$ such that
$$\left|\left(\mathcal{F}\left[\int_{\sigma\in \widehat{G}_{\text{temp},i\mathfrak{h}^*_+}}\theta_{\sigma}d\mu_{\pi}'\right]*t^{-n/4}\mathcal{F}[\varphi](t^{-1/2}\cdot)\right)(t\eta)\right|\leq C_{N,\epsilon}t^{-N}.$$
Here $\mathcal{F}$ denotes the Fourier transform and $n=\dim G$. Taking this Fourier transform, the left hand side becomes
$$\left(\int_{\sigma\in \widehat{G}_{\text{temp},i\mathfrak{h}^*_+}}\mathcal{O}_{\sigma} d\mu_{\pi}' * t^{-n/4}\mathcal{F}[\varphi](t^{-1/2}\cdot)\right) (t\eta)$$
$$=\int_{\sigma\in \widehat{G}_{\text{temp},i\mathfrak{h}^*_+}}\left(\int_{\mathcal{O}_{\sigma}}t^{-n/4}\mathcal{F}[\varphi]\left(\frac{t\eta-\zeta}{\sqrt{t}}\right)d(\mathcal{O}_{\sigma})_{\zeta}\right)d\mu_{\pi}'.$$
Thus, to prove a contradiction and conclude that $\xi$ is indeed in the wave front set, we must find a suitable measure $\mu_{\pi}'$, a constant $C$, and an integer $M$ such that 
$$\left|\int_{\sigma\in \widehat{G}_{\text{temp},i\mathfrak{h}^*_+}}\left(\int_{\mathcal{O}_{\sigma}}t_m^{-n/4}\mathcal{F}[\varphi]\left(\frac{t_m\eta_m-\zeta}{\sqrt{t_m}}\right)d(\mathcal{O}_{\sigma})_{\zeta}\right)d\mu_{\pi}'\right|\geq Ct_m^{-M}$$
for a sequence $(t_m,\eta_m)$ with $\eta_m\in \mathcal{C}$, $||\xi|-|\eta_m||<\epsilon$, and $t_m\rightarrow \infty$.

To do this, we first take our open cone $\mathcal{C}$, and we note that there exists $\delta<\epsilon$ such that $\mathcal{C}\supset \mathcal{C}_{\delta}$ where
$$\mathcal{C}_{\delta}=\{\eta\in i\mathfrak{g}^*|\ |\xi-t\eta|<\delta\ \text{some}\ t>0\}.$$
Since $\xi\in \operatorname{AC}(\mathcal{O}\operatorname{-}\operatorname{supp}\pi)$, we know that $(\mathcal{O}\operatorname{-}\operatorname{supp}\pi)\cap \mathcal{C}_{\delta}$ is noncompact. Therefore, we may find a sequence $\{t_m\eta_m\}$ inside this intersection such that $t_m>t_{m-1}+2$ and $|\eta_m|=1$ for every $m$.

Let $\mathcal{O}_{t_m\eta_m}=\mathcal{O}_{\sigma_m}$ and for $\sigma_m'$ near $\sigma_m$, consider the set 
$$S_{m,\sigma_m'}=\{\zeta\in \mathcal{O}_{\sigma_m'}\cap \mathcal{C}_{\delta}|\ ||\zeta|-|t_m\eta_m||<1\}.$$
Let $$F_m(\sigma_m')=\langle \operatorname{Eucl}(\mathcal{O}_{\sigma_m'}),S_{m,\sigma_m'}\rangle$$
be the volume of this set with respect to the Euclidean measure induced on the corresponding orbit. Since $t_m\eta_m\in \mathcal{C}_{\delta/2}$ and $t_m\delta/2\geq 1$ for sufficiently large $m$, we deduce that $F_m(\sigma_m)\geq t_m^{-k_1}$ for sufficiently large $m$ and some $k_1>0$. Since $F_m(\sigma_m')$ is a continuous function of $\sigma_m'$, we can find a neighborhood $N_m$ of $\sigma_m$ in $\widehat{G}_{\text{temp},i\mathfrak{h}^*_+}$ for each $m$ such that $F_m(\sigma_m')\geq (1/2)t_m^{-k_1}$ for every $\sigma_m'\in N_m$. In addition, observe that the sets
$$\bigcup_{\sigma'\in N_m} S_{m,\sigma_m'}$$
are disjoint.

Now, since $\sigma_m$ is in the support of $\mu_{\pi}$ and $N_m$ is an open neighborhood containing $\sigma_m$, we must have $$\mu_{\pi}(N_m)>0.$$ We may choose a positive, finite measure $\mu_{\pi}'$ equivalent to $\mu_{\pi}$ for which 
$$\int_{N_m} \mu_{\pi}'\geq t_m^{-M_0}$$
for some fixed, sufficiently large integer $M_0$.

Next, we must estimate 
$$F_m'(\sigma_m')=\langle m(\mathcal{O}_{\sigma_m'}),S_{m,\sigma_m'}\rangle$$
from $F_m$ where the measure on the orbit is now the canonical invariant measure. To estimate this volume, we use Lemma \ref{canonicalmeasure}. Recall that we wrote
$$F(\eta)m(\mathcal{O}_{\eta})_{\eta}=\operatorname{Eucl}(\mathcal{O}_{\eta})_{\eta}$$
By Lemma \ref{canonicalmeasure}, there exist constants $C>0$ and $N>0$ such that 
$$F(\eta)\geq C(1+(t_m-1))^{-N}=Ct_m^{-N}$$
whenever $\eta\in S_{m,\sigma_m'}$ with $\sigma_m'\in N_m$.
Thus, we obtain $$F_m'(\sigma'_m)\geq Ct_m^{-N}F_m(\sigma_m')\geq (C/2)t_m^{-N-k_1}.$$
Putting all of this together, we estimate

\begin{align*} & \left|\int_{\sigma\in \widehat{G}_{\text{temp},i\mathfrak{h}^*_+}}\left(\int_{\mathcal{O}_{\sigma}}t_m^{-n/4}\mathcal{F}[\varphi]\left(\frac{t_m\eta_m-\zeta}{\sqrt{t_m}}\right)d(\mathcal{O}_{\sigma})_{\zeta}\right)d\mu_{\pi}'\right| \\
& \geq \left|\int_{\sigma\in N_m}\left(\int_{S_{m,\sigma}}t_m^{-n/4}\mathcal{F}[\varphi]\left(\frac{t_m\eta_m-\zeta}{\sqrt{t_m}}\right)d(\mathcal{O}_{\sigma})_{\zeta}\right)d\mu_{\pi}'\right| \\
& \geq \left|\int_{\sigma\in N_m}\left(\int_{S_{m,\sigma}}t_m^{-n/4}\cdot 1d(\mathcal{O}_{\sigma})_{\zeta}\right)d\mu_{\pi}'\right| \\
& \geq \left(\int_{\sigma\in N_m}d\mu_{\pi}'\right) \cdot t_m^{-n/4}\cdot \langle m(\mathcal{O}_{\sigma}),S_{m,\sigma}\rangle \\
& \geq (C/2) t_m^{-M_0-N-k_1-n/4}.
\end{align*}
This is what we needed to prove. The Proposition now follows.
\end{proof}

\section{Wave Front Sets of Pieces of the Regular Representation Part II}
\label{sec:regular_part2}

As explained in the beginning of the last section, we now prove the second inclusion necessary for the proof of Theorem \ref{regularintro}.

\begin{proposition} \label{ssregular} If $G$ is a real, reductive algebraic group and $\pi$ is weakly contained in the regular representation of $G$, then
$$\operatorname{SS}(\pi)\subset \operatorname{AC}\left(\mathcal{O}\operatorname{-}\operatorname{supp}\pi\right).$$
\end{proposition}

Utilizing Lemma \ref{Lem:WeylChamber_Reduction}, we may assume $\operatorname{supp}\pi\subset \widehat{G}_{\text{temp},i\mathfrak{h}^*_+}$
for a fixed Weyl chamber $i\mathfrak{h}^*_+$ in $i$ times the dual of a fixed Cartan subalgebra $\mathfrak{h}\subset \mathfrak{g}$. We will make this assumption throughout this section.

First, we require a technical Lemma. Suppose $W$ is a finite dimensional, real vector space, and suppose $0\in U_1\subset U_2\subset W$ are precompact, open subsets of $W$ with $\overline{U_1}\subset U_2$. Recall from Section 2 that we have fixed a family of functions $\varphi_{N,U_1,U_2}$ satisfying certain properties including (\ref{eq:varphi_family_bound}).

\begin{lemma} \label{boundaryvalues} Suppose $W$ is a finite-dimensional real vector space, suppose $\widetilde{W}$ is an open neighborhood of zero in another finite-dimensional real vector space, and suppose we have an analytic map $$\psi\colon \widetilde{W}\times W\rightarrow W$$ such that for each $p\in \widetilde{W}$, $\psi_p$ is locally bianalytic and $\psi_0=I$ is the identity. Suppose $u$ is a distribution on $W$, suppose $(x,\xi)\notin \operatorname{SS}(u)$, and suppose $a$ is an analytic function on $W$. Then one can find an open set $0\in \widetilde{U}\subset \widetilde{W}$, an open set $\xi\in \Omega\subset iW^*$, and an open set $x\in U_2\subset W$ such that for every pair of precompact open sets $x\in U_1\subset U\subset U_2\subset W$ with $U_1$ compactly contained in $U$, there exists a constant $C_{U_1,U}>0$ such that 
$$\left|\mathcal{F}\left[a\left(\psi_p^*u\right)\varphi_{N,U_1,U}\right](t\eta)\right|\leq C_{U_1,U}^{N+1}(N+1)^Nt^{-N}$$
whenever $p\in \widetilde{U}$, $\eta\in \Omega$, and $t>0$.
\end{lemma}

The thing that makes this Lemma non-trivial is the uniformity of the bound in the variable $p\in \widetilde{U}$. We will prove it by relating the singular spectrum to boundary values of analytic functions, utilizing Sections 8.4 and 8.5 of \cite{Hor}.

\begin{proof} Since $\operatorname{SS}(u)\subset iT^*W$ is a closed set, we may choose an open set $x\in U_3\subset W$ and an open cone $\xi\in \mathcal{C}(1)$ such that $U_3\times \mathcal{C}(1)\subset iT^*W-\operatorname{SS}(u)$. Next, fix an open cone $$\xi\in \mathcal{C}(2)\subset \overline{\mathcal{C}(2)}\subset \mathcal{C}(1).$$ If $\mathcal{C}\subset W$ is an open convex cone, we may form the dual cone 
$$\mathcal{C}^0=\{\xi\in iW^*|\ i\langle \xi,y\rangle\leq 0\ \forall\ y\in W\}.$$ If $\eta\in iW^*-\overline{\mathcal{C}(2)}$, we may find a cone of the form $\mathcal{C}^0$, which is the dual cone of an open convex cone $\mathcal{C}$, such that $\eta\in \mathcal{C}^0\subset iW^*-\overline{\mathcal{C}(2)}$. Fixing an inner product on the finite-dimensional real vector space $W$ and using the compactness of $\mathbb{S}^{\dim W-1}\cap (iW^*-\mathcal{C}(1))$, we may choose a finite subcover $\mathcal{C}_1^0,\ldots,\mathcal{C}_k^0$ of $iW^*-\mathcal{C}(1)$. Here each $\mathcal{C}_j^0$ is the dual cone of an open convex cone $\mathcal{C}_j$. In particular, we have
$$\bigcup_{w\in U_3} \operatorname{SS}_w(u)\subset U_3\times \bigcup_{i=1}^k \mathcal{C}_i^0,\ \xi\in iW^*-\bigcup_{i=1}^k \overline{\mathcal{C}_i^0}.$$
Now, in addition, we may choose $(\mathcal{C}_j')^0$ such that $\mathcal{C}_j^0$ is contained in the interior of $(\mathcal{C}_j')^0$, $\xi\notin (\mathcal{C}_j')^0$ for any $j$, and $(\mathcal{C}_j')^0$ is the dual cone of an open convex cone $\mathcal{C}_j'\subset \mathcal{C}_j$. 
Utilizing Corollary 8.4.13 of \cite{Hor}, we may write $u=\sum_{j=1}^k u_j$ with 
$$\bigcup_{w\in U_3} \operatorname{SS}_w(u_j)\subset U_3\times \mathcal{C}_j^0$$ We note that to obtain the estimate in the Lemma for the distribution $u$, it is enough to obtain the estimate for each distribution $u_j$. 

Next, choose $x\in U_4\subset U_3$ an open subset with $\overline{U_4}\subset U_3$ a compact subset. If $\gamma>0$ is a real number, define $$\mathcal{C}_j(\gamma)=\{y\in \mathcal{C}_j|\ |y|<\gamma\}.$$ By the remark after Theorem 8.4.15 of \cite{Hor}, for some $\gamma_j>0$, we may find an analytic function $F_j$ in $U_4+i\mathcal{C}_j(\gamma_j)\subset W_{\mathbb{C}}$, where $W_{\mathbb{C}}=W\otimes_{\mathbb{R}}\mathbb{C}$ is the complexification of the vector space $W$, such that $F_j$ satisfies an estimate
$$|F_j(x+iy)|\leq C_j|y|^{-N_j}$$
in $U_4+i\mathcal{C}_j(\gamma)$ and 
$$u_j=\lim_{y\rightarrow 0,\ y\in \mathcal{C}_j(\gamma_j)}F_j(\cdot+iy).$$
Here the limit is taken in the space of distributions on $W$. Next, we may complexify the maps $\psi_p$ to attain the map
$$\psi_{\mathbb{C}}\colon \widetilde{W}\times W_{\mathbb{C}}\rightarrow W_{\mathbb{C}}$$
which is real analytic in the first variable and complex analytic in the second. Taylor expand each $\psi_{\mathbb{C}}$ at $(0,x)\in \widetilde{W}\times W_{\mathbb{C}}$ as a function of $v\in W_{\mathbb{C}}$ with coefficients that are real analytic functions in $p\in \widetilde{W}$. One sees from this expansion that we may find open sets $x\in U_2\subset U_4$ and $0\in \widetilde{U}\subset \widetilde{W}$ together with positive constants $\gamma_j'>0$ such that 
$$\psi_p(U_2+i\mathcal{C}_j'(\gamma_j'))\subset U_4+i\mathcal{C}_j(\gamma_j)$$
for every $p\in \widetilde{U}$ and every $j=1,\ldots,k$. After possibly shrinking $U_2$, $\widetilde{U}$ and decreasing $\gamma_j'$, we see from the Taylor expansion that we may in addition assume 
$$|y|/2\leq |\operatorname{Im}\psi_p(x+iy)|\leq 2|y|$$
for all $p\in \widetilde{U}$, $x+iy\in U_2+i\mathcal{C}_j'(\gamma_j')$. From now on, we will write $u_j$ for the restriction of $u_j$ to $U_2$ and $F_j$ for the restriction of $F_j$ to $U_2+i\mathcal{C}_j'(\gamma_j')$. As in the proof of Theorem 8.5.1 of \cite{Hor}, we now have
$$\psi_p^*u_j=\lim_{y\rightarrow 0,\ y\in \mathcal{C}_j'(\gamma_j')}\psi_p^*F_j(\cdot+iy)$$
for $p\in \widetilde{U}$ and $j=1,\ldots,k$. In addition, we obtain the bounds
$$|(\psi_p^*F_j)(x+iy)|\leq 2^NC_j|y|^{-N_j}=C_j'|y|^{-N_j}$$
uniform in $p\in \widetilde{U}$. 

Of course, we may multiply through by our analytic function $a$ to obtain
$$a \psi_p^*u_j=\lim_{y\rightarrow 0,\ y\in \mathcal{C}_j'(\gamma_j')}a \psi_p^*F_j(\cdot+iy)$$
for $p\in \widetilde{U}$ and $j=1,\ldots,k$ and 
$$|a(\psi_p^*F_j)(x+iy)|\leq C_j''|y|^{-N_j}$$
uniform in $p\in \widetilde{U}$. 

Now, we use these uniform bounds on $a\psi_p^*F_j$ to obtain uniform bounds on $$\mathcal{F}\left[a\left(\psi_p^*u\right)\varphi_{N,U_1,U}\right].$$ To do this, we utilize the proof of Theorem 8.4.8 of \cite{Hor}. We observe that the constant $C_4$ in (8.4.9) on the top of page 286 of \cite{Hor} depends only on the constants $C_j'$, $N_j$ in the above bound on $a\psi^*F_j$ and on the functions $\varphi_{N,U_1,U}$. Since these constants are uniform in $p$, we obtain the necessary bounds on $\mathcal{F}\left[a\left(\psi_p^*u\right)\varphi_{N,U_1,U}\right]$ uniform in $p\in \widetilde{U}$ and the Lemma has been proven.
\end{proof}

Now, suppose $(\pi,V)$ is a unitary representation of a real, reductive algebraic group $G$ that is weakly contained in the regular representation. Decompose
$$\pi\cong \int_{\sigma\in \widehat{G}_{\text{temp}}} \sigma^{\oplus m(\pi,\sigma)}d\mu_{\pi}$$
into irreducibles. As noted before, by Lemma \ref{Lem:WeylChamber_Reduction}, we may assume without loss of generality that $\mu_{\pi}$ is a finite, positive measure on $\widehat{G}_{\text{temp},i\mathfrak{h}^*_+}$ for a fixed Weyl chamber $i\mathfrak{h}^*_+$. By Lemma \ref{tempereddist}, if $\mu_{\pi}$ is a finite positive measure on $\widehat{G}_{\text{temp},i\mathfrak{h}^*_+}$ and $f\in L^1(\widehat{G}_{\text{temp},i\mathfrak{h}^*_+},\mu_{\pi})$, then the integral 
$$\int_{\sigma\in \widehat{G}_{\text{temp},i\mathfrak{h}^*_+}} \theta_{\sigma} f(\sigma) d\mu_{\pi}$$
defines a tempered distribution on $i\mathfrak{g}^*$. Moreover, by the proof of Lemma \ref{tempereddist}, we see that the Fourier transform of this tempered distribution is
$$\int_{\sigma\in \operatorname{supp}\pi} \mathcal{O}_{\sigma} f(\sigma) d\mu_{\pi}.$$
Clearly this distribution is supported in $\mathcal{O}\operatorname{-}\operatorname{supp}\pi$. Therefore, by Lemma 8.4.17 on page 194 of \cite{Hor}, we deduce 
\begin{equation}\label{eq:SS_bound}
\operatorname{SS}_0\left(\int_{\sigma\in \operatorname{supp}\pi} \theta_{\sigma} f(\sigma) d\mu_{\pi}\right)\subset \operatorname{AC}(\mathcal{O}\operatorname{-}\operatorname{supp}\pi)
\end{equation}
whenever $\mu_{\pi}$ is a finite, positive measure associated to the representation $\pi$.

Next, we prove Proposition \ref{ssregular}. In this argument, we will set $W=\mathfrak{g}$ and form sequences of the form $\varphi_{N,U_1,U}$ satisfying the properties outlined in Section \ref{sec:def_wf} including (\ref{eq:varphi_family_bound}).  
\begin{proof} As in the statement of Proposition \ref{ssregular}, fix a unitary representation $(\pi,V)$ of a real, reductive algebraic group $G$ that is weakly contained in the regular representation. Choose $\xi\notin \operatorname{AC}(\mathcal{O}\operatorname{-}\operatorname{supp}\pi)$. We must show that for every $u,v\in V$, there exists an open set $\xi\in \Omega\subset i\mathfrak{g}^*$ and a constant $C>0$ such that

$$\left|\int_{G}(\pi(g)u,v)\varphi_{N,U_1,U}(g)e^{t\eta(\log g)}dg\right|\leq C^{N+1}(N+1)^Nt^{-N}$$
for every $\eta\in \Omega$ and $t>0$.

For each $\sigma\in \widehat{G}_{\text{temp.}}$, we abuse notation and write $(\sigma,V_{\sigma})$ for a representative of this equivalence class of irreducible tempered representations. We have a direct integral decomposition
$$V\cong \int_{\operatorname{supp}\pi} V_{\sigma}^{\oplus m(\pi,\sigma)}d\mu_{\pi}(\sigma).$$

Now, if $u=(u_{\sigma})$ and $v=(v_{\sigma})$ in our direct integral decompositions, then we have
$$(\pi(g)u,v)=\int_{\sigma\in \operatorname{supp}\pi}(\sigma(g)u_{\sigma},v_{\sigma})d\mu_{\pi}(\sigma).$$
Thus our integral becomes
\begin{align*} & \left|\int_{G}(\pi(g)u,v)\varphi_{N,U_1,U}(g)e^{t\eta(\log g)}dg\right|\\
& =\left|\int_{G}\varphi_{N,U_1,U}(g)e^{t\eta(\log g)}\int_{\sigma\in \operatorname{supp}\pi}(\sigma(g)u_{\sigma},v_{\sigma})d\mu_{\pi}(\sigma)dg\right| \\
& =\left|\int_{\sigma\in \operatorname{supp}\pi}\int_{G}\varphi_{N,U_1,U}(g)e^{t\eta(\log g)}(\sigma(g)u_{\sigma},v_{\sigma})dgd\mu_{\pi}(\sigma)\right| \\
& =\left|\int_{\sigma\in \operatorname{supp}\pi}(\sigma(\varphi_{N,U_1,U}e^{t\eta(\log)})u_{\sigma},v_{\sigma})\mu_{\pi}(\sigma)\right| \\
& \leq \int_{\sigma\in \operatorname{supp}\pi}\left|\sigma(\varphi_{N,U_1,U}e^{t\eta(\log)})u_{\sigma}\right| \cdot \left|v_{\sigma}\right|d\mu_{\pi}(\sigma)\\
& \leq \left(\int_{\sigma\in \operatorname{supp}\pi} \left|\sigma(\varphi_{N,U_1,U}e^{t\eta(\log)})\right|_{HS}^2 \cdot \left|u_{\sigma}\right|^2d\mu_{\pi}(\sigma)\right)^{1/2}\left(\int_{\sigma\in \operatorname{supp}\pi} |v_{\sigma}|^2d\mu_{\pi}(\sigma)\right)^{1/2}.
\end{align*}

Here $|\cdot|_{HS}$ denotes the Hilbert-Schmidt norm of an operator on $V_{\sigma}$. Moreover, we are abusing notation and writing $\sigma(\varphi_{N,U_1,U})$ for the action of $\varphi_{N,U_1,U}$ on $V_{\sigma}\otimes M_{\sigma}$ as well as $V_{\sigma}$. The second integral is a constant. Therefore, we may focus on the first integral.

Next, we use a calculation of Howe (see page 128 of \cite{How}). For $\sigma\in \widehat{G}_{\text{temp}}$, we have
$$\int_G \overline{\varphi_{N,U_1,U}}(g^{-1})e^{\eta(\log)}\langle \Theta_{\sigma}, l_g[\varphi_{N,U_1,U}e^{\eta(\log)}]\rangle dg=|\sigma(\varphi_{N,U_1,U}e^{\eta(\log)})|_{HS}^2.$$
Integrating both sides over $\sigma\in \operatorname{supp}\pi$ with respect to $|u_{\sigma}|^2 d\mu_{\pi}(\sigma)$ yields
$$\int_G \overline{\varphi_{N,U_1,U}}(g^{-1})e^{\eta(\log)}\langle \int_{\sigma\in \operatorname{supp}\pi}\Theta_{\sigma}|u_{\sigma}|^2d\mu_{\pi}(\sigma), l_g[\varphi_{N,U_1,U}e^{\eta(\log)}]\rangle dg$$
$$=\int_{\sigma\in \operatorname{supp}\pi}|\sigma(\varphi_{N,U_1,U}e^{\eta(\log)})|_{HS}^2|u_{\sigma}|^2d\mu_{\pi}(\sigma).$$
We observe that getting the proper bounds for the right hand side is what we need in order to prove our Proposition. We will obtain them by bounding the left hand side utilizing Lemma \ref{boundaryvalues} together with the remarks afterwards.

Choose an open set $0\in \widetilde{V}\subset \mathfrak{g}$ such that $\exp\colon \widetilde{V}\rightarrow \exp(\widetilde{V})$ is a bianalytic isomorphism onto its image. We apply Lemma \ref{boundaryvalues} with $V=\mathfrak{g}$, $\widetilde{V}$ as above,
$$\psi\colon \widetilde{V}\times \mathfrak{g}\rightarrow \mathfrak{g}$$
by $(Y,X)\mapsto \log(\exp Y\exp X)$, $a=j_G^{1/2}$, and  
$$u=\int_{\sigma\in \operatorname{supp}\pi} \theta_{\sigma}|u_{\sigma}|^2d\mu_{\pi}(\sigma).$$
Moreover, utilizing (\ref{eq:SS_bound}) with $f(\sigma)=|u_{\sigma}|^2$, we obtain $(0,\xi)\notin \operatorname{SS}_0(u)$ since by hypothesis $\xi\notin \operatorname{AC}(\mathcal{O}\operatorname{-}\operatorname{supp}\pi)$. Then Lemma \ref{boundaryvalues} assures us of the existence of open sets $0\in \log(U_1)\subset \log(U)\subset \log(\widetilde{U})\subset \widetilde{V}$ such that the closure of $\log(U_1)$ is contained in the interior of $\log(U)$ together with an open set $\xi\in \Omega\subset i\mathfrak{g}^*$ and a constant $C>0$ such that

$$\left|\int_{\sigma\in \operatorname{supp}\pi} \left(\int_{\mathfrak{g}} j_G^{1/2}(X)\theta_{\sigma}(\exp Y \exp X) (\exp^*\varphi_{N,U_1,U})(X)e^{t\eta(X)}dX\right) |u_{\sigma}|^2d\mu_{\pi}(\sigma)\right|$$
$$\leq C^{N+1}(N+1)^Nt^{-N}$$
whenever $\eta\in \Omega$, $Y\in \log(\widetilde{U})$, and $t>0$.

Pulling back to the group, we obtain
$$\left|\int_{\sigma\in \operatorname{supp}\pi} \left(\int_{G} \Theta_{\sigma}(gh) \varphi_{N,U_1,U}(h)e^{t\eta(\log(h))}dh\right) |u_{\sigma}|^2d\mu_{\pi}(\sigma)\right|$$
$$\leq C(C(N+1))^Nt^{-N}$$
whenever $g\in \widetilde{U}$, $\eta\in \Omega$, and $t>0$. Substituting and changing the order of integration yields
$$\left|\langle \int_{\sigma\in \operatorname{supp}\pi} \Theta_{\sigma}|u_{\sigma}|^2d\mu_{\pi}(\sigma), l_g\left[\varphi_{N,U_1,U}(h)e^{t\eta(\log(h))}\right]\rangle \right|$$
$$\leq C^{N+1}(N+1)^Nt^{-N}$$
whenever $g\in \widetilde{U}$, $\eta\in \Omega$, and $t>0$. Finally, if we integrate over $g$ in a precompact set with respect to a smooth density multiplied by a bounded function, this will simply multiply the bound by a constant, which we may incorporate into $C$. Thus, we obtain
$$\left|\int_G \overline{\varphi_{N,U_1,U}}(g^{-1})e^{\eta(\log)}\langle \int_{\sigma\in \operatorname{supp}\pi}\Theta_{\sigma}|u_{\sigma}|^2\mu_{\pi}(\sigma), l_g[\varphi_{N,U_1,U}e^{\eta(\log)}]\rangle dg\right|$$
$$\leq C^{N+1}(N+1)^Nt^{-N}$$
for $\eta\in \Omega$ and $t>0$. Tracing back through our calculations, we see that we obtain
$$\left|\int_{G}(\pi(g)u,v)\varphi_{N,U_1,U}(g)e^{t\eta(\log g)}dg\right|\leq C^{(N+1)/2}(N+1)^{N/2} t^{-N/2}$$
for $\eta\in \Omega$ and $t>0$. We simply replace $N$ by $2N$ and note that the sequence $\varphi_{2N,U_1,U}$ still satisfies the necessary conditions needed for Definition \ref{ssdefh}. Then we obtain
$$\left|\int_{G}(\pi(g)u,v)\varphi_{2N,U_1,U}(g)e^{t\eta(\log g)}dg\right|\leq (C')^{N+1}(N+1)^{N} t^{-N}$$
for $\eta\in \Omega$ and $t>0$. Proposition \ref{ssregular} and Theorem \ref{regularintro} now follow.
\end{proof}

\section{Examples and Applications}
\label{sec:examples_apps}

In this section, we will give examples of our results in the case $G=\operatorname{SL}(2,\mathbb{R})$. Then we will briefly mention applications to branching problems and harmonic analysis questions.

\subsection{Wave Front Sets of Representations of $G=\operatorname{SL}(2,\mathbb{R})$}
\label{sl(2)}

First, we consider the special case of the group $G=\operatorname{SL}(2,\mathbb{R})$. We identify $\mathfrak{g}=\mathfrak{sl}(2,\mathbb{R})$ with $\mathbb{R}^3$ via
$$(x,y,z)\mapsto \left(\begin{matrix} x & y-z\\ y+z & -x \end{matrix}\right).$$ 
In addition, we identify $\mathfrak{g}\cong \mathfrak{g}^*$ using the trace form,
$$X\mapsto (Y\mapsto \operatorname{Tr}(XY)).$$
Dividing by $i$, we obtain a (non-canonical) isomorphism $i\mathfrak{g}^*\cong \mathbb{R}^3$ which is useful for drawing pictures.
The coadjoint orbits of $\operatorname{SL}(2,\mathbb{R})$ come in several classes. First, we have the hyperbolic orbits, 
$$\mathcal{O}_{\nu}=\{(x,y,z)|\ x^2+y^2-z^2=\nu^2\}$$
for $\nu>0$. Next, we have two classes of elliptic orbits,
$$\mathcal{O}_n^+=\{(x,y,z)|\ z^2-x^2-y^2=n^2,\ z>0\},$$
$$\mathcal{O}_n^-=\{(x,y,z)|\ z^2-x^2-y^2=n^2,\ z<0\}$$
for any real number $n>0$. Then we have the two large pieces of the nilpotent cone
$$\mathcal{N}^+=\{(x,y,z)|\ x^2+y^2=z^2,\ z>0\},$$
$$\mathcal{N}^-=\{(x,y,z)|\ x^2+y^2=z^2,\ z<0\}.$$
And finally we have the zero orbit, $\{0\}$.

The irreducible, unitary representations of $\operatorname{SL}(2,\mathbb{R})$ also come in several classes. First, we have the spherical unitary principal series $\operatorname{\sigma}_{\nu,+}$ for $\nu\geq 0$ as well as the non-spherical unitary principal series $\sigma_{\nu,-}$ for $\nu>0$. Next, we have the holomorphic discrete series representations $\sigma_n^+$ and the antiholomorphic discrete series representations $\sigma_n^-$ for $n\in \mathbb{Z}_{>0}$. Here we have parametrized the discrete series by infinitesimal character. In addition, the terms `holomorphic' and `antiholomorphic' come from the standard holomorphic structure on the upper half plane and the standard identification of $\operatorname{SL}(2,\mathbb{R})/\operatorname{SO}(2,\mathbb{R})$ with the upper half plane. Finally, we have the limits of discrete series, $\sigma^+$ and $\sigma^-$.

There is also the trivial representation of $\operatorname{SL}(2,\mathbb{R})$ as well as the complementary series, but these representations are not tempered; hence, we will not consider them in this paper. 

Now, the representations $\sigma_{\nu,+}$ and $\sigma_{\nu,-}$ are associated to the orbit $\mathcal{O}_{\nu}$ for $\nu>0$, and the representation $\sigma_{0,+}$ is associated to the nilpotent cone $\mathcal{N}=\mathcal{N}^+\cup \mathcal{N}^-\cup \{0\}$. The representation $\sigma_n^+$ (respectively $\sigma_n^-$) is associated to the orbit $\mathcal{O}_n^+$ (respectively $\mathcal{O}_n^-$). And the representation $\sigma^+$ (respectively $\sigma^-$) is associated to the orbit $\mathcal{N}^+$ (respectively $\mathcal{N}^-$). 

Next, we utilize Theorem \ref{regularintro} to compute the wave front sets of some representations. One notes
$$\operatorname{WF}(\sigma_n^+)=\operatorname{AC}(\mathcal{O}_n^+)=\mathcal{N}^+,$$
$$\operatorname{WF}(\sigma_n^-)=\operatorname{AC}(\mathcal{O}_n^-)=\mathcal{N}^-,$$
$$\operatorname{WF}(\sigma_{\nu,+})=\operatorname{WF}(\sigma_{\nu,-})=\operatorname{AC}(\mathcal{O}_{\nu})=\mathcal{N}$$
for $\nu>0$. In addition,
$$\operatorname{WF}(\sigma_{0,+})=\operatorname{AC}(\mathcal{N})=\mathcal{N},$$
$$\operatorname{WF}(\sigma^+)=\operatorname{AC}(\mathcal{O}^+)=\mathcal{N}^+,$$
$$\operatorname{WF}(\sigma^-)=\operatorname{AC}(\mathcal{O}^-)=\mathcal{N}^-.$$

Of course, all of these computations of wave front sets of irreducible, unitary representations have been well-known for sometime because of the work of Barbasch-Vogan \cite{BV} and Rossmann \cite{R5}. What is new in this paper is our ability to compute wave front sets of representations that are far from irreducible. 

Suppose $A\subset \operatorname{SL}(2,\mathbb{R})$ is the set of diagonal matrices. Utilizing Theorem \ref{inducedintro}, we observe
$$\operatorname{WF}(L^2(\operatorname{SL}(2,\mathbb{R})/A))\supset \overline{\operatorname{Ad}^*(G)\cdot i(\mathfrak{g}/\mathfrak{a})^*}=i\mathfrak{g}^*.$$
Therefore, $\operatorname{WF}(L^2(\operatorname{SL}(2,\mathbb{R})/A))=i\mathfrak{sl}(2,\mathbb{R})^*$. Similarly, if $\Gamma\subset \operatorname{SL}(2,\mathbb{R})$ is a discrete subgroup, then 
$$\operatorname{WF}(L^2(\operatorname{SL}(2,\mathbb{R})/\Gamma))=i\mathfrak{g}^*.$$
Of course, one could deduce these first two facts from Theorem \ref{regularintro} together with the well-known decomposition of $L^2(G/A)$ and the existence of sufficiently many Poincare series and Eisenstein series for $\Gamma$. However, the authors like that we are able to compute these wave front sets without knowledge of these decompositions.

Next, we utilize Theorem \ref{regularintro}. Let $i\mathfrak{g}_{\text{hyp}}^*$ denote the set of hyperbolic elements in $i\mathfrak{g}^*$. Identifying $i\mathfrak{g}^*$ with $\mathbb{R}^3$ as above, we have
$$i\mathfrak{g}^*_{\text{hyp}}=\{(x,y,z)|\ x^2+y^2-z^2>0\}.$$ Let $i\mathfrak{g}^*_{\text{ell}}$ denote the set of elliptic elements in $i\mathfrak{g}^*$. Break this set up into two by
$$(i\mathfrak{g}^*_{\text{ell}})^+=\{(x,y,z)|\ z^2-x^2-y^2>0,\ z>0\},$$
$$(i\mathfrak{g}^*_{\text{ell}})^-=\{(x,y,z)|\ z^2-x^2-y^2>0,\ z<0\}.$$
If $K=\operatorname{SO}(2,\mathbb{R})$, then we have 
$$\operatorname{WF}\left(L^2(G/K)\right)=\operatorname{WF}\left(\int_{\nu>0} \sigma_{\nu,+}\right)=\operatorname{AC}\left(\bigcup_{\nu>0} \mathcal{O}_{\nu}\right)=\overline{i\mathfrak{g}^*_{\text{hyp}}}.$$
Similarly, we have
$$\operatorname{WF}\left(\int_{\nu>0} \sigma_{\nu,-}\right)=\overline{i\mathfrak{g}^*_{\text{hyp}}}.$$
In addition, we have
$$\operatorname{WF}\left(\bigoplus_{n>0} \sigma_n^+\right)=\operatorname{AC}\left(\bigcup_{n>0} \mathcal{O}_n^+\right)=\overline{(i\mathfrak{g}^*_{\text{ell}})^+},$$
$$\operatorname{WF}\left(\bigoplus_{n>0} \sigma_n^-\right)=\operatorname{AC}\left(\bigcup_{n>0} \mathcal{O}_n^-\right)=\overline{(i\mathfrak{g}^*_{\text{ell}})^-}.$$

\subsection{Wave Front Sets and Branching Problems}
\label{branching}

Next, we say a few words about branching problems. We recall the statement of Corollary \ref{resintro1}.

Suppose $G$ is a real, reductive algebraic group, suppose $H\subset G$ is a closed reductive algebraic subgroup, and suppose $\pi$ is a discrete series representation of $G$. Let $\mathfrak{g}$ (resp. $\mathfrak{h}$) denote the Lie algebra of $G$ (resp. $H$), and let $q\colon i\mathfrak{g}^*\rightarrow i\mathfrak{h}^*$ be the pullback of the inclusion. Then 
$$\operatorname{AC}(\mathcal{O}\operatorname{-}\operatorname{supp}(\pi|_H))\supset q(\operatorname{WF}(\pi)).$$

This Corollary follows directly from Theorem \ref{regularintro}, Proposition 1.5 of \cite{How}, and the fact that the restriction of a discrete series to a reductive subgroup is weakly contained in the regular representation (see for instance Theorem 3 of \cite{OV}, though this is neither the first nor the easiest proof of this fact). 

We show how to utilize this Corollary in a simple example. First, let $G=\operatorname{SU}(2,1)$ and let $H=\operatorname{SO}(2,1)_e\cong \operatorname{PSL}(2,\mathbb{R})$ be the identity component of the subgroup of $G$ consisting of real matrices. If $\pi$ is a quaternionic discrete series of $G$, then one can show $$\operatorname{WF}(\pi)=\mathcal{N}_G$$
where $\mathcal{N}_G$ is the nilpotent cone in $i\mathfrak{g}^*$. One checks via a simple linear algebra calculation that $\overline{q(\mathcal{N}_G)}=i\mathfrak{h}^*$. Thus, we obtain
$$\operatorname{AC}(\mathcal{O}\operatorname{-}\operatorname{supp}(\pi|_H))= i\mathfrak{h}^*.$$
One notes that the irreducible, tempered representations of $\operatorname{PSL}(2,\mathbb{R})$ are the irreducible, tempered representations of $\operatorname{SL}(2,\mathbb{R})$ in which the center of $\operatorname{SL}(2,\mathbb{R})$ acts trivially. These are the spherical unitary principal series and ``half'' of the holomorphic and antiholomorphic discrete series representations.
We then deduce that
\begin{itemize}
\item $\pi|_H$ contains an integral of spherical unitary principal series with unbounded support.
\item $\pi|_H$ contains infinitely many distinct holomorphic discrete series.
\item $\pi|_H$ contains infinitely many distinct antiholomorphic discrete series. 
\end{itemize}

The authors believe that the last two facts are non-trivial. For comparison, one can see utilizing arguments in \cite{OO} that whenever $\pi$ is a holomorphic discrete series of $G$, the restriction $\pi|_H$ contains at most finitely many holomorphic and antiholomorphic discrete series representations. 

Next, we recall the statement of Corollary \ref{resintro2}.  Suppose $G$ is a real, reductive algebraic group, $H\subset G$ is a closed reductive algebraic subgroup, and $\pi$ is a discrete series representation of $G$. Let $\mathfrak{g}$ (resp. $\mathfrak{h}$) denote the Lie algebra of $G$ (resp. $H$), and let $q\colon i\mathfrak{g}^*\rightarrow i\mathfrak{h}^*$ be the pullback of the inclusion. If $\pi|_H$ is a Hilbert space direct sum of irreducible representations of $H$, then 
$$q(\operatorname{WF}(\pi))\subset \overline{i\mathfrak{h}^*_{\text{ell}}}.$$
Here $i\mathfrak{h}^*_{\text{ell}}\subset i\mathfrak{h}^*$ denotes the subset of elliptic elements.

This statement follows from Corollary \ref{resintro1} together with the fact that only discrete series of $H$ can occur discretely in $\pi|_H$ when $\pi$ is a discrete series of $G$ (see Corollary 8.7 on page 131 of \cite{Ko5}) and the fact that discrete series correspond to elliptic caodjoint orbits \cite{R1}. 

To illustrate Corollary \ref{resintro2}, we consider tensor products of discrete series representations of $\operatorname{SL}(2,\mathbb{R})$. This particular example has been well understood for a long time (see \cite{Re}). We use it because it is simple and it illustrates our ideas well. 

The exterior tensor product $\sigma_n^+\boxtimes \sigma_m^+$ (resp. $\sigma_n^-\boxtimes \sigma_m^-$) corresponds to the product of orbits $\mathcal{O}_n^+\times \mathcal{O}_m^+$ (resp. $\mathcal{O}_n^-\times \mathcal{O}_m^-$) as a representation of $\operatorname{SL}(2,\mathbb{R})\times \operatorname{SL}(2,\mathbb{R})$. The projection 
$$i\operatorname{sl}(2,\mathbb{R})^*\oplus i\operatorname{sl}(2,\mathbb{R})^*\rightarrow i\operatorname{sl}(2,\mathbb{R})^*$$
is given by the sum $(\xi,\eta)\mapsto \xi+\eta$. One checks that
$$\mathcal{O}_n^++\mathcal{O}_m^+\subset (i\mathfrak{g}^*_{\text{ell}})^+,\ \ \mathcal{O}_n^-+\mathcal{O}_m^-\subset (i\mathfrak{g}_{\text{ell}}^*)^-.$$
In fact, $\sigma_n^+\otimes \sigma_m^+$ is a discrete sum of holomorphic discrete series and $\sigma_n^-\otimes \sigma_m^-$ is a discrete sum of antiholomorphic discrete series (see Theorem 1 and Example 5 of \cite{Re}). Therefore, the Corollary told us that these sums of orbits would be contained in the elliptic set.

On the other hand, the exterior tensor product $\sigma_n^+\boxtimes \sigma_m^-$ corresponds to the product of orbits $\mathcal{O}_n^+\times \mathcal{O}_m^-$. Their sum contains the set of hyperbolic elements $i\mathfrak{g}^*_{\text{hyp}}$. Utilizing the contrapositive of Corollary \ref{resintro2}, we deduce that $\sigma_n^+\otimes \sigma_m^-$ is not a discrete sum of irreducible representations. In fact, utilizing Corollary \ref{resintro1}, one deduces that it must contain an unbounded integral of unitary principal series. One checks that this is the case (see Theorem 2 and Example 5 of \cite{Re}). 

\subsection{Wave Front Sets and Harmonic Analysis}
\label{harmonic_analysis}

Next, we consider applications to harmonic analysis questions. Recall Corollary \ref{BKintro}. If $L^2(G/H)$ is weakly contained in the regular representation, then 
$$\operatorname{AC}(\mathcal{O}\operatorname{-}\operatorname{supp}L^2(G/H))\supset \overline{\operatorname{Ad}^*(G)\cdot i(\mathfrak{g}/\mathfrak{h})^*}.$$
We need several remarks on how to use this result. First, it will be helpful to introduce the following notation.
If $\mathfrak{h}\subset \mathfrak{g}$ is a Cartan subalgebra and $\pi$ is a representation of $G$ that is weakly contained in the regular representation, then we define
$$i\mathfrak{h}^*\operatorname{-}\operatorname{supp}\pi=\bigcup_{\sigma\in \operatorname{supp}\pi} (\mathcal{O}_{\sigma}\cap i\mathfrak{h}^*)\subset i\mathfrak{h}^*$$
We note that only irreducible, tempered representations with regular infinitesimal character contribute to $i\mathfrak{h}^*\operatorname{-}\operatorname{supp}\pi\cap (i\mathfrak{h}^*)'$, where $(i\mathfrak{h}^*)'$ denotes the set of regular elements in $i\mathfrak{h}^*$. Further, any irreducible, tempered representation with regular infinitesimal character contributes exactly one orbit of a real Weyl group in $i\mathfrak{h}^*$ for a Cartan subalgebra $\mathfrak{h}\subset \mathfrak{g}$, unique up to conjugacy by $G$.

We deduce from the above discussion that for $\pi$ weakly contained in the regular representation
$$\overline{\operatorname{AC}(\mathcal{O}\operatorname{-}\operatorname{supp}\pi)\cap (i\mathfrak{h}^*)'}\subset \operatorname{AC}(i\mathfrak{h}^*\operatorname{-}\operatorname{supp}\pi)\subset \operatorname{AC}(\mathcal{O}\operatorname{-}\operatorname{supp}\pi)\cap i\mathfrak{h}^*.$$
In particular, if $\operatorname{AC}(\mathcal{O}\operatorname{-}\operatorname{supp}\pi)=i\mathfrak{g}^*$, then 
$$\operatorname{AC}(i\mathfrak{h}^*\operatorname{-}\operatorname{supp}\pi)=i\mathfrak{h}^*$$ for every Cartan subalgebra $\mathfrak{h}\subset \mathfrak{g}$. The authors feel that this is ample justification for saying that $\operatorname{supp}\pi$ is \emph{asymptotically identical} to $\operatorname{supp}L^2(G)$ if 
$$\operatorname{AC}(\mathcal{O}\operatorname{-}\operatorname{supp}\pi)=i\mathfrak{g}^*.$$

Second, we recall the recent work of Benoist and Kobayashi \cite{BK}. Suppose $G$ is a real, reductive algebraic group, and suppose $H\subset G$ is a real, reductive algebraic subgroup. Let $\mathfrak{g}$ (resp. $\mathfrak{h}$) denote the Lie algebras of $G$ (resp. $H$). Let $\mathfrak{a}\subset \mathfrak{h}$ be a maximal split abelian subspace, and recall that we have Lie algebra maps $\mathfrak{a}\rightarrow \operatorname{End}(\mathfrak{h})$ and $\mathfrak{a}\rightarrow \operatorname{End}(\mathfrak{g})$ given by the adjoint actions. If $Y\in \mathfrak{a}$, define $\mathfrak{h}_{+,Y}$ (resp. $\mathfrak{g}_{+,Y}$) to be the sum of the positive eigenspaces for the adjoint action of $Y$ on $\mathfrak{h}$ (resp. $\mathfrak{g}$), and define
$$\rho_{\mathfrak{h}}(Y)=\operatorname{Tr}_{\mathfrak{h}_{+,Y}}(Y),\ \rho_{\mathfrak{g}}(Y)=\operatorname{Tr}_{\mathfrak{g}_{+,Y}}(Y).$$
In Theorem 4.1 of \cite{BK}, Benoist and Kobayashi show that $L^2(G/H)$ is weakly contained in the regular representation of $G$ if and only if
$$2\rho_{\mathfrak{h}}(Y)\leq \rho_{\mathfrak{g}}(Y)\ \text{for\ every}\ Y\in \mathfrak{a}.$$
Now, suppose $H\subset G$ are real, reductive algebraic groups satisfying the above condition. Then Corollary \ref{BKintro} implies
$$\operatorname{AC}(\mathcal{O}\operatorname{-}\operatorname{supp}L^2(G/H))\supset\overline{\operatorname{Ad}^*(G)\cdot i(\mathfrak{g}/\mathfrak{h})^*}.$$
We note that the right hand side is quite computable. Let $\mathfrak{q}$ be the orthogonal complement of $\mathfrak{h}$ with respect to a nondegenerate, invariant form (the Killing form will due if $G$ is simple). After dividing by $i$ and identifying $\mathfrak{g}^*$ with $\mathfrak{g}$ via this form, we need only ask ``which elements of $\mathfrak{g}$ are conjugate to elements of $\mathfrak{q}$'' in order to compute the right hand side of the above expression. In particular, the right hand side is $i\mathfrak{g}^*$ if $\mathfrak{q}$ contains representatives of every conjugacy class of Cartan subalgebra in $\mathfrak{g}$.

Benoist and Kobayashi give large families of examples of pairs $H\subset G$ satisfying their condition in Example 5.6 and Example 5.10 of \cite{BK}. We will focus on Example 5.6. We see that if $G=\operatorname{SO}(p,q)$ and $H=\prod_{i=1}^r \operatorname{SO}(p_i,q_i)$ with $p=\sum_{i=1}^r p_i$, $q=\sum_{i=1}^r q_i$, and $2(p_i+q_i)\leq p+q+2$ whenever $p_iq_i\neq 0$, then $L^2(G/H)$ is weakly contained in the regular representation. To the best of the authors' knowledge, Plancherel formulas are not known for the vast majority of these cases. One checks using parametrizations of conjugacy classes of Cartan subalgebras (see \cite{Kos}, \cite{Su}) and an explicit description of $\mathfrak{q}$, that if in addition, $2p_i\leq p+1$,$2q_i\leq q+1$ for every $i$ and $p+q>2$, then $$i\mathfrak{g}^*=\overline{\operatorname{Ad}^*(G)\cdot i(\mathfrak{g}/\mathfrak{h})^*}.$$
Utilizing Corollary \ref{BKintro}, we deduce $\operatorname{supp} L^2(G/H)$ is asymptotically equivalent to $\operatorname{supp}L^2(G)$. In particular, suppose $p$ and $q$ are not both odd and $\mathcal{F}$ is one of the families of discrete series of $G=\operatorname{SO}(p,q)$ corresponding to a fixed Weyl chamber. Under these assumptions, if $\mathfrak{h}_0$ is a compact Cartan subalgebra of $\mathfrak{g}$, then we observe
$$\operatorname{AC}(i\mathfrak{h}_0^*\operatorname{-}\operatorname{supp}L^2(G/H))=i\mathfrak{h}_0^*.$$
In particular, we deduce that for every family $\mathcal{F}$ of discrete series of $G$,
$$\operatorname{Hom}_G(\sigma,L^2(G/H))\neq \{0\}$$
for infinitely many different $\sigma\in \mathcal{F}$.
A particularly nice example is when $G=\operatorname{SO}(4n,2)$ and $H=\operatorname{SO}(n,1)\times\operatorname{SO}(n,1)\times \operatorname{SO}(2n)$. In this case, one deduces
$$\operatorname{Hom}_G(\sigma,L^2(G/H))\neq \{0\}$$
for infinitely many distinct (possibly vector valued) holomorphic discrete series $\sigma$ of $G$. We note that when $n=1$, this statement can be deduced from Theorem 7.5 on page 126 of Kobayashi's paper \cite{Ko5}. 

\section{Acknowledgements}
The authors would like to thank Todd Quinto for providing us with advice and references for fundamental facts in microlocal analysis. The authors would also like to thank David Vogan for a few comments on a previous draft.

\bibliographystyle{amsplain}

\begin{thebibliography}{65}
\bibitem{ALTV} J. Adams, M. van Leuwen, P. Trapa, and D. Vogan, {\em Unitary Representations of Real Reductive Groups}, arXiv: 1212.2192.
\bibitem{BV} D. Barbasch, D. Vogan, {\em The Local Structure of Characters}, Journal of Functional Analysis, Volume 37, Number 1, 27-55 (1980)
\bibitem{BK} Y. Benoist, T. Kobayashi, {\em Temperedness of Reductive Homogeneous Spaces}, To Appear in Journal of the European Mathematical Society. Preprint: arXiv:1211.1203.
\bibitem{Bo} J.M. Bony, {\em Equivalence des diverses notions de spectre singulier analytique}, S\'{e}m. Goulaouic-Schwartz (1976-1977), \'{E}quations aux d\'{e}riv\'{e}es partielles et analyse fonctionnelle, Expos\'{e} no. III, 12 pp. Centre Math., \'{E}cole Polytech., Palaiseau (1977)
\bibitem{CF} A. Cordoba, C. Fefferman, {\em Wave Packets and Fourier Integral Operators}, Comm. Partial Diff. Eq., Volume 3, 979-1005 (1978)
\bibitem{CdS} A. Cannas da Silva, {\em Lecture Notes on Symplectic Geometry}, Lecture Notes in Mathematics, Volume 1764, Springer-Verlag, Berlin, (2001)
\bibitem{Di} J. Dixmier, {\em $C^*$-algebras}, North Holland Press, Amsterdam (1977)
\bibitem{Du} M. Duflo, {\em  Fundamental Series Representations of a Semisimple Lie Group}, Functional Analysis Applications, Volume 4, 122-126 (1970)
\bibitem{Fo} G.B. Folland, {\em Harmonic Analysis in Phase Space}, Princeton University Press, Princeton, NJ (1989)
\bibitem{FJ} G. Friedlander, M. Joshi, {\it Introduction to the Theory of Distributions}, Cambridge University Press, Cambridge (1998)
\bibitem{Fu} H. Fujiwara, {\em Sur les restrictions des repr\'{e}sentations unitaires des groupes de Lie r\'{e}solubles exponentiels}, Inventiones Mathematicae, Volume 104, Number 3, 647-654 (1991)
\bibitem{HHK} S. Hansen, J. Hilgert, S. Keliny, {\it Asymptotic $K$-support and Restrictions of Representations}, Representation Theory, Volume 13,  460-469 (2009)
\bibitem{HC3} Harish-Chandra {\em Representations of Semisimple Lie Groups III}, Transactions of the American Mathematical Society, Volume 76, 234-253 (1954)
\bibitem{HC} Harish-Chandra {\it Fourier Transforms on a Semisimple Lie Algebra I}, American Journal of Mathematics, Volume 79, No. 2, 193-257 (1957)
\bibitem{HC58a} Harish-Chandra, {\em Spherical Functions on a Semisimple Lie Group I}, American Journal of Mathematics, Volume 80, 553-613 (1958) 
\bibitem{HC58b} Harish-Chandra, {\em Spherical Functions on a Semisimple Lie Group II}, American Journal of Mathematics, Volume 80, 241-310 (1958)
\bibitem{HC2} Harish-Chandra, {\em Some Results on an Invariant Integral on a Semi-Simple Lie Algebra}, Annals of Mathematics Series 2, Volume 80, 551-593 (1964)
\bibitem{HC65} Harish-Chandra, {\em Discrete Series for Semisimple Lie Groups I: Construction of Invariant Eigendistributions}, Acta Mathematica, Volume 113, 241-318 (1965)
\bibitem{he} H. He, {\em Generalized Matrix Coefficients for  Unitary Representations}, To appear in Journal of Ramanujan Mathematical Society.
\bibitem{HS} H. Hecht, W. Schmid, {\em A Proof of Blattner's Conjecture}, Inventiones Mathematicae, Volume 31, 129-154 (1975)
\bibitem{He70} S. Helgason, \emph{A Duality for Symmetric Spaces with Applications to Group Representations}, Advances in Mathematics,
Volume 5, 1-154 (1970)
\bibitem{He73} S. Helgason, \emph{Functions on Symmetric Spaces}, In ``Harmonic Analysis on Homogeneous Spaces'' (Proc.
Sympos. Pure Math., Vol. XXVI, Williams Coll., Williamstown, Mass., 1972), 101–146, American Mathematical Society, Providence, R.I. (1973)
\bibitem{Hor1} L. H\"{o}rmander, {\em Fourier Integral Operators I}, Acta Mathematica, Volume 127, 79-183 (1971)
\bibitem{Hor} L. H\"{o}rmander, {\em The Analysis of Linear Partial Differential Operators I}, Grundlehren der mathematischen Wissenschaften, Volume 256, Springer-Verlag, Berlin (1983)
\bibitem{Hor2} L. H\"{o}rmander, {\em The Analysis of Linear Partial Differential Operators III}, Grundlehren der mathematischen Wissenschaften, Volume 274, Springer-Verlag, Berlin (1985)
\bibitem{How} R. Howe, {\em Wave Front Sets of Representations of Lie Groups}, In ``Automorphic Forms, Representation Theory, and Arithmetic'' (Bombay 1979), Tata Inst. Fund. Res. Studies in Math., 10, 117–140, Tata Inst. Fundamental Res., Bombay (1981)
\bibitem{Ia} D. Iagolnitzer, {\em Microlocal Essential Support of a Distribution and Decomposition Theorems--An Introduction.} In Hyperfunctions and Theoretical Physics, Springer Lecture Notes in Math 449, 121-132, Spinger-Verlag, Berlin (1975)
\bibitem{KKS} M. Kashiwara, T. Kawai, M. Sato, {\em Microfunctions and Pseudo-Differential Equations}, In ``Hyperfunctions and Pseudo-Differential Equations'', (Proceedings of the Conference Katata 1971), pages 265-529, Lecture Notes in Mathematics, Volume 287, Springer-Verlag, Berlin (1973)
\bibitem{KV} M. Kashiwara, M. Vergne, {\it $K$-types and Singular Spectrum}, Noncommutative Harmonic Analysis (Proc. Third Colloq., Marseilles-Luminy, 1978), Lecture Notes in Mathematics, Volume 728, Springer Verlag, Berlin, 177-200 (1979)
\bibitem{Kn} A. Knapp, {\em Representation Theory of Semisimple Lie Groups}, Princeton University Press, Princeton, NJ (1986)
\bibitem{KnZ} A. Knapp, G. Zuckerman, {\em Classification of Irreducible Tempered Representations of Semisimple Lie Groups}, Annals of Mathematics Series 2, Volume 118, 389-501 (1982)
\bibitem{Ko2} T. Kobayashi, {\em Discrete decomposability of the restriction of $A_{\mathfrak{q}}(\lambda)$ with respect to reductive subgroups. II. Micro-local analysis and asymptotic K-support}, Annals of Mathematics Series 2, Volume 147, Number 3, 709-729 (1998)
\bibitem{Ko3} T. Kobayashi, {\em Discrete decomposability of the restriction of $A_{\mathfrak{q}}(\lambda)$ with respect to reductive subgroups. III. Restriction of Harish-Chandra modules and associated varieties}, Inventiones Mathematicae, Volume 131, Number 2, 229-256 (1998)
\bibitem{Ko5} T. Kobayashi, \emph{Discrete series representations for the orbit space arising from two involutions of real reductive Lie groups}, Journal of Functional Analysis \textbf{152} (1998), no. 1, 100--135.
\bibitem{Ko4} T. Kobayashi, {\em Branching Problems of Zuckerman Derived Functor Modules.} Pages 23-40 of Representation Theory and Mathematical Physics (in honor of Gregg Zukermann), edited by Adams, Lian, and Sahi, Volume 557 of Contemporary Mathematics, American Mathematical Society, Providence, RI (2011)
\bibitem{Kos} B. Kostant, {\em On the Conjugacy of Real Cartan Subalgebras I}, Proceedings of the National Academy of Sciences USA, Volume 41, 967-970 (1955)
\bibitem{La} R. Langlands, {\em On the Classification of Irreducible Representations of Real Algebraic Groups}, In ``Representation Theory and Harmonic Analysis on Semisimple Lie Groups'', Mathematical Surveys and Monographs, Volume 31, American Mathematical Society, Providence, RI (1989)
\bibitem{Mo} M. Morimoto, {\em An Introduction to Sato's Hyperfunctions}, Translations of Mathematical Monographs, Volume 129, American Mathematical Society (1993)
\bibitem{OO} G. \'{O}lafsson, B. \O rsted, {\it Generalizations of the Bargmann Transform}, In ``Lie Theory and Its Applications to Physics'' (Clausthal 1995), 3-14, World Sci. Publ., River Edge, NJ (1996)
\bibitem{OS} G. \'{O}lafsson, H. Schlichtkrull, {Representation Theory, Radon Transform and the Heat Equation on a Riemannian Symmetric Space}, In ``Group Representations, Ergodic Theory, and Mathematical Physics: a Tribute to George W. Mackey'', 315–344, Contemporary Mathematics, Volume 449, American Mathematical Society, Providence, RI (2008)
\bibitem{OV} B. \O rsted, J. Vargas, {\it Restriction of Square Integrable Representations: Discrete Spectrum}, Duke Mathematical Journal, Volume 123, Number 3, 609-633 (2004)
\bibitem{Re} J. Repka, {\em Tensor Products of Holomorphic Discrete Series Representations}, Canadian Journal of Mathematics, Volume 31, Number 4, 836-844 (1979)
\bibitem{R1} W. Rossmann, {\it Kirillov's Character Formula for Reductive Lie Groups}, Inventiones Mathematicae, Volume 48, Number 3, 207-220 (1978)
\bibitem{R2} W. Rossmann, {\it Limit Characters of Reductive Lie Groups}, Inventiones Mathematicae, Volume 61, Number 1, 53-66 (1980)
\bibitem{R3} W. Rossmann, {\it Limit Orbits in Reductive Lie Algebras}, Duke Mathematical Journal, Volume 49, Number 1, 215-229 (1982)
\bibitem{R5} W. Rossmann, {\em Picard-Lefschetz Theory and Characters of a Semisimple Lie Group}, Inventiones Mathematicae, Volume 121, Number 3, 579-611 (1995)
\bibitem{Sa} M. Sato, {\em Regularity of Hyperfunction Solutions of Partial Differential Equations}, Volume 2 of Actes Congr. Int. Math., (Nice, 1970), 785-794 (1971)
\bibitem{SV} W. Schmid, K. Vilonen, {\em Characteristic Cycles and Wave Front Cycles of Representations of Reductive Lie Groups}, Annals of Mathematics Series 2, Volume 151, Number 3, 1071-1118 (2000)
\bibitem{Su} M. Sugiura, {\em Conjugate Classes of Cartan Subalgebras in Real Semisimple Lie Algebras}, Journal of the Mathematical Society of Japan, Volume 11, Number 4, 374-434 (1959)
\bibitem{V2} D. Vogan, {\em Representations of Real Reductive Lie Groups}, Progress in Mathematics, Volume 15, Birkhauser, Boston, MA (1981) 
\bibitem{V} D. Vogan, {\em Associated Varieties and Unipotent Representations}, Harmonic Analysis on Reductive Groups (Brunswick, ME, 1989), Progress in Mathematics 101, 315-388, Birkhauser, Boston, MA (1991)

\end{thebibliography}

\end{document}